\documentclass[12pt]{article}
\usepackage[UKenglish]{babel}
\usepackage[T1]{fontenc}
\usepackage{lmodern,amsmath,amsthm,amsfonts,amssymb,graphicx,float,microtype,thmtools,underscore,mathtools,thm-restate}
\usepackage[shortlabels]{enumitem}
\setlist[itemize]{topsep=0ex,itemsep=0ex,parsep=0ex}
\setlist[enumerate]{topsep=0ex,itemsep=0ex,parsep=0ex}
\usepackage[usenames,dvipsnames,svgnames,table]{xcolor}
\usepackage{todonotes}
\usepackage[unicode=true]{hyperref}
\usepackage{breakurl}
\hypersetup{ 
colorlinks,
linkcolor={blue!60!black},
citecolor={black},
urlcolor={blue!60!black},
pdftitle={Path Independence Number}}
\usepackage[capitalise, compress, nameinlink, noabbrev]{cleveref}
\crefname{lem}{Lemma}{Lemmas}
\crefname{thm}{Theorem}{Theorems}
\crefname{ques}{Question}{Theorems}
\crefname{cor}{Corollary}{Corollaries}
\crefname{enumi}{Item}{Items}
\crefformat{equation}{(#2#1#3)}
\Crefformat{equation}{Equation #2(#1)#3}
\crefformat{enumi}{#2#1#3}
\Crefformat{enumi}{Item (#2#1#3)}
\newcommand{\defn}[1]{\textcolor{Maroon}{\emph{#1}}}
\usepackage[longnamesfirst,numbers,sort&compress]{natbib}
\makeatletter
\def\NAT@spacechar{~}
\makeatother
\usepackage[margin=28mm]{geometry}
\renewcommand{\baselinestretch}{1.1}
\allowdisplaybreaks
\DeclarePairedDelimiter{\ceil}{\lceil}{\rceil}

\renewcommand{\epsilon}{\varepsilon}
\renewcommand{\emptyset}{\varnothing}
\renewcommand{\geq}{\geqslant}
\renewcommand{\leq}{\leqslant}
\DeclareMathOperator{\dist}{dist}
\DeclareMathOperator{\tw}{tw}
\DeclareMathOperator{\atw}{\alpha-tw}
\DeclareMathOperator{\apw}{\alpha-pw}

\DeclareMathOperator{\pw}{pw}

\DeclareMathOperator{\asep}{sep_{\alpha}}

\newcommand{\BB}{\mathcal{B}}

\newcommand{\GG}{\mathcal{G}}
\newcommand{\HH}{\mathcal{H}}

\newcommand{\NN}{\mathbb{N}}

\renewcommand{\thefootnote}{\fnsymbol{footnote}}
\theoremstyle{plain}
\newtheorem{thm}{Theorem}
\newtheorem{lem}[thm]{Lemma}

\newtheorem{prop}[thm]{Proposition}

\newtheorem{obs}[thm]{Observation}
\newtheorem{claim}{Claim}
\crefname{obs}{Observation}{Observations}
\newtheorem*{lem*}{Lemma}
\theoremstyle{definition}
\newtheorem{conj}[thm]{Conjecture}
\newtheorem*{conj*}{Conjecture}
\presetkeys%
{todonotes}%
{inline}{}
\date{}

\begin{document}

\title{\bf\fontsize{18pt}{18pt}\selectfont Induced Forest Minor Theorem for Graphs Without an Induced Star}

\author{Robert Hickingbotham\,\footnotemark[1]	\quad \quad \quad Gwena\"el Joret\,\footnotemark[1]	}

\footnotetext[1]{D\'epartement d'Informatique, Universit\'e libre de Bruxelles, Belgium ({\tt robert.hickingbotham@ulb.be}, {\tt gwenael.joret@ulb.be}). R.\ Hickingbotham and G.\ Joret are supported by the Belgian National Fund for Scientific Research (FNRS).}

\maketitle
\begin{abstract}
Motivated by recent work on tree independence number, we study the \emph{path independence number} of a graph $G$: the minimum integer $k$ such that there is a path decomposition of $G$ where each bag induces a graph with independence number at most $k$. We show that every graph excluding both an induced forest minor and an induced star has bounded path independence number. This characterises when a graph class that excludes an induced star has bounded path independence number while also partially resolving a conjecture of Dallard, Krnc, Kwon, Milani{\v{c}}, Munaro, {\v{S}}torgel and Wiederrecht (2024). Furthermore, we show that graphs excluding both an apex-forest induced minor and an induced star have bounded tree independence number. As a consequence, for every fixed apex-forest $H$ and integer $t$, there is a polynomial-time algorithm to test whether a $K_{1,t}$-induced-subgraph-free graph contains $H$ as an induced minor. Moreover, it follows that the Maximum Weight Independent Set problem, as well as several other NP-hard problems, can be solved in polynomial-time on $K_{1,t}$-induced-subgraph-free graphs that exclude $H$ as an induced minor. 
\end{abstract}

\renewcommand{\thefootnote}{\arabic{footnote}}

\section{Introduction}

A graph $H$ is an \defn{(induced) minor} of a graph $G$ if it can be obtained from an (induced) subgraph of $G$ by contracting edges.\footnote{All graphs in this paper are finite and simple. See \cref{SectionPrelim} for undefined terms and notation.}

A \defn{tree-decomposition} of a graph $G$ is a collection $(B_x\subseteq V(G): x \in V(T))$ of \emph{bags} indexed by a tree $T$ such that (i) for each $v \in V(G)$, $T[\{x \in V(T): v \in B_x\}]$ is a non-empty subtree of $T$; and (ii) for each $uv \in E(G)$, there is a node $x \in V(T)$ such that $u,v \in B_x$. A \defn{path-decomposition} is a tree-decomposition indexed by a path. The \defn{width} of a tree-decomposition is $\max\{|B_x|-1\colon x \in V(T)\}$. The \defn{treewidth} $\tw(G)$ (\defn{pathwidth} $\pw(G)$) of a graph $G$ is the minimum width of a tree-decomposition (path-decomposition) of $G$. Treewidth and pathwidth are important parameters in algorithmic and structural graph theory.

The \emph{Graph Minors} series of \citet{GraphMinors} established a rich theory of the structure of graphs defined by an excluded minor. Among the central results are the Forest Minor Theorem~\cite{robertson1983graph}, which states that graphs that exclude a forest as a minor have bounded pathwidth, and the Grid Minor Theorem~\cite{robertson1986planar}, which states that graphs that exclude a planar graph as a minor have bounded treewidth. These theorems characterise when a graph class has bounded pathwidth or bounded treewidth, respectively.

Inspired by the success of Robertson and Seymour's graph minor theory, a recent line of work aims to establish an analogous theory for induced minors. Currently, far
less is known about classes excluding an induced minor compared to those excluding a minor. Induced-minor-closed classes may be dense and have proven to have a far more complex structure than minor-closed classes~\cite{BH2025rigs}. Nevertheless, in the bounded-degree setting, \citet{korhonen2022} established the following Induced Grid Minor Theorem.

\begin{thm}[\cite{korhonen2022}]\label{KorhonenInducedGrid}
    For every planar graph $H$ and $\Delta\in \NN$, there exists $k\in \NN$ such that every graph $G$ with maximum degree at most $\Delta$ and no $H$ induced minor has $\tw(G)\leq k$.
\end{thm}

Since treewidth is a sparse parameter (i.e. bounded treewidth graphs have bounded average degree), an additional sparsity constraint such as bounded-degree is necessary to ensure bounded treewidth. However, such sparsity constraints diminish the appeal of induced minors which is their suitability to the study of dense graphs. This motivates the search for an induced minor analog of treewidth that can be small on dense graphs while having similar structural and algorithmic properties to treewidth.

In recent years, a growing body of evidence suggests that tree independence number is the right parameter. The \defn{tree independence number} $\atw(G)$ of a graph $G$ is the minimum integer $k$ such that $G$ admits a tree-decomposition where each bag has independence number at most $k$. Introduced independently by \citet{Yolov2018} and \citet{DALLARD2024indep}, this parameter is closed under taking induced minors \cite{DALLARD2024indep}, exhibits rich structural and algorithmic properties \cite{LMMORS2024matching}, while presenting applications in coarse graph theory \cite{hickingbotham2025quasi,NSS2025asymptoticI}. It is easy to see that complete graphs have tree-independence number $1$, so unlike treewidth, this parameter is suitable for the study of dense graphs.

One of the key open problems concerning tree independence number is whether it admits a grid-like theorem. In this direction, \citet{DKKMSW2024TI4} conjectured the following induced grid minor theorem for graphs with \emph{bounded local independence number}; that is, graphs where the neighbourhood of every vertex has bounded independence number. For graphs $H$ and $G$, we say that $G$ is \defn{$H$-free} if it excludes $H$ as an induced subgraph.

\begin{conj}[\cite{DKKMSW2024TI4}]\label{ConjTreeAlpha}
    For every planar graph $H$ and $t\in \NN$, there exists $k\in \NN$ such that every $K_{1,t}$-free graph $G$ with no $H$ induced minor has $\atw(G)\leq k$.
\end{conj}

One can readily observe that $K_{1,t}$-freeness is equivalent to the neighbourhood of every vertex having independence number less than $t$. Similarly, having maximum degree at most $\Delta$ is equivalent to excluding $K_{1,\Delta+1}$ as a subgraph. Thus \cref{ConjTreeAlpha} would generalise \cref{KorhonenInducedGrid} from treewidth to tree independence number. \cref{ConjTreeAlpha} has been verified for when $H$ is a disjoint union of triangles \cite{AGHK2025EP}, a wheel \cite{CHMW2025wheel}, an induced subgraph of the $1$-subdivision of a ladder \cite{CW2025Ladder}, or an outerstring graph~\cite{BCKV26}.  
For an arbitrary planar graph $H$,  a weakening of \cref{ConjTreeAlpha} with a polylogarithmic bound on the tree independence number is known to hold~\cite{chudnovsky2025treeindependencenumberviiexcluding}.

Our main result resolves \cref{ConjTreeAlpha} when $H$ is a forest in a qualitatively stronger sense. The \defn{path independence number} $\apw(G)$ of a graph $G$ is the minimum integer $k$ such that $G$ admits a path-decomposition where every bag has independence number at most $k$. Like tree independence number, path independence number is closed under taking induced minor. This parameter was first considered in~\cite{BCKMMS2024simultaneous}.

\begin{restatable}{thm}{MainThmForest}\label{MainThmForest}
	For every forest $H$ and $t\in \NN$, there exists $k\in \NN$ such that every $K_{1,t}$-free graph $G$ with no $H$ induced minor has $\apw(G)\leq k$.
\end{restatable}

\cref{MainThmForest} demonstrates that path independence number is a natural induced-minor analog of pathwidth. Since the class of forests is triangle-free with unbounded pathwidth, Ramsey's Theorem~\cite{ramsey1930problem} implies that it has unbounded path independence number. Thus, we have the following consequence of \cref{MainThmForest}.

\begin{thm}
    For every $t\in \NN$, for every class $\GG$ of $K_{1,t}$-free graphs, $\GG$ has bounded path independence number if and only if there is a forest $F$ such that every graph in $\GG$ is $F$-induced-minor-free.
\end{thm}

A graph $H$ is an \defn{apex-forest} if there is a vertex $v\in V(H)$ such that $H-v$ is a forest.  Our second contribution is the following.

\begin{thm}\label{MainThmApexForest}
    For every apex-forest $H$ and $t\in \NN$, there exists $k\in \NN$ such that every $K_{1,t}$-free graph $G$ with no $H$ induced minor has $\atw(G)\leq k$.
\end{thm}

Before discussing the applications of our main results, we briefly survey the literature on `Induced Forest Minor Theorems'. Let $F$ be a forest. Using \citet{korhonen2022} contraction--uncontraction technique, \citet{Hickingbotham22} showed that bounded-degree graphs with no $F$ induced minor have bounded pathwidth. \citet{CHS2024pathwidth} subsequently showed that $K_t$-free graphs excluding both $F$ and $K_{t,t}$ as induced minors have bounded pathwidth. This breakthrough result allows them to fully classify the unavoidable induced subgraphs of graphs with large pathwidth. \citet{NSS2025fat} showed that graphs with no $F$ induced minor have a path decomposition where each bag consists of a bounded number of balls of bounded radius (where the number and size of the balls depend on $F$).\footnote{\citet{NSS2025fat} proved a more general result concerning fat minors which implies this statement for induced minors.} This implies that there exists a forest $F'$ such that $F$-induced-minor-free graphs are quasi-isometric to graphs with no $F'$ minor. Finally, \citet{BDH2025Forest} classified the forests $H$ for which every $K_{t,t}$-subgraph-free $H$-induced-minor-free graph has bounded pathwidth.

\subsubsection*{Applications}

The \textsc{$H$-Induced Minor Containment} problem asks, for a fixed graph $H$, whether a given graph $G$ contains $H$ as an induced minor. Unlike the \textsc{$H$-Minor Containment} problem which is solvable in polynomial-time, this problem can be NP-complete for a fixed graph $H$~ \cite{FKMP1995complexity}. In particular, it was shown by \citet{KL2024InducedMinor} that for some trees $H$, the \textsc{$H$-Induced Minor Containment} is \textsc{NP}-complete. See \cite{DDHP2025inducedminor} for a survey on this problem.

\citet{CHMW2025wheel} showed that for every planar graph $H$ and integer $t\in \NN$ that satisfies \cref{ConjTreeAlpha}, there is a polynomial-time algorithm that, given a $K_{1,t}$-free graph $G$ as input, either decides that $G$ does not contain $H$ as an induced minor or finds an induced minor model of $H$.\footnote{\citet{CHMW2025wheel} only established this result for when $H$ is a wheel. It is easy to see that their proof generalises to any graph $H$ that satisfies \cref{ConjTreeAlpha}.} Thus, \cref{MainThmApexForest} implies the following:

\begin{thm}
    For every apex-forest $H$ and $t\in \NN$, there exists a polynomial-time algorithm that, given a $K_{1,t}$-free graph $G$, either decides that $G$ does not contain $H$ as an induced minor or outputs an induced minor model of $H$ in $G$.
\end{thm}

\citet{Yolov2018} and \citet{DALLARD2024indep} independently showed that, for every fixed $k\in \NN$, there exists a polynomial-time algorithm for solving the \textsc{Maximum Weight Independent Set} problem on any graph $G$ with $\atw(G)\leq k$. Note that the relevant decomposition does not need to be given as input since it can be approximated efficiently \cite{Yolov2018,DFKM2024aprox}. Combining this with \cref{MainThmApexForest} yields the following.

\begin{thm}
    For every apex-forest $H$ and integer $t\in \NN$, there exists a polynomial-time algorithm for solving \textsc{Maximum Weight Independent Set} on any $K_{1,t}$-free graph $G$ that does not contain $H$ as an induced minor.
\end{thm}

See \cite{LMMORS2024matching} for other \textsc{NP}-hard problems that are polynomial time solvable on graphs with bounded tree independence number.  

\subsubsection*{Independence Erd\H{o}s--P\'osa}

The classical Erd{\H{o}}s-P{\'o}sa theorem~\cite{ErdosPosa1965} asserts that for every integer $k$, every graph $G$ contains $k$ vertex disjoint cycles, or there exists a set $X$ of vertices with $|X|\in \mathcal{O}(k\log(k))$ that intersects every cycle in $G$. This seminal result initiated a long line of research into Erd{\H{o}}s-P{\'o}sa-type theorems, where the aim is to show that a graph contains many vertex-disjoint subgraphs from a specified class, or a small hitting set intersecting all such subgraphs; see \cite{RT17} for a survey.

In recent years, attention has turned to induced Erd{\H{o}}s-P{\'o}sa-type theorems, where `disjoint' is replaced by `anti-complete' \cite{AGHK2025EP,DJMM2024EP}. For a positive integer $d$ and graph $H$, let $dH$ denote the disjoint union of $d$ copies of $H$. \citet{CW2025Ladder} defined a graph $H$ to have the \defn{independence Erd\H{o}s--P\'osa property} in $K_{1,t}$-free graphs if there exists a function $f$ such that, for every $K_{1,t}$-free graph $G$ and positive integer $d$, $G$ contains $dH$ as an induced minor (i.e. $d$ pairwise anti-complete induced subgraphs each containing $H$ as an induced minor), or there exists a set $X\subseteq V(G)$ with $\alpha(X)\leq f(d,t)$ such that $G-X$ does not contain $H$ as an induced minor. They proved that any induced subgraph of the $1$-subdivision of a ladder has this property. In particular, their argument shows that if $dH$ satisfies \cref{ConjTreeAlpha} for all $d$, then $H$ has this property. We therefore obtain the following consequence of \cref{MainThmForest}.

\begin{thm}
    Every forest has the independence Erd\H{o}s--P\'osa property in $K_{1,t}$-free graphs.
\end{thm}

We conclude this section by mentioning the following very recent related result of \citet{chudnovsky2026forbiddinganticompleteplanarminors}: 
For every planar graph $H$ there exists a function $f$ such that for every positive integer $d$ and every graph $G$ not containing $d$ pairwise anti-complete  subgraphs each containing $H$ as a minor, there is $X\subseteq V(G)$ with $|X|\leq f(d)$ such that $G-N_G[X]$ does not contain $H$ as a minor. 
It was conjectured by \citet{CMPRR26} that this result holds more generally for induced minors: For every planar graph $H$ there exists a function $f$ such that for every positive integer $d$ and every graph $G$ not containing $dH$ as an induced minor, there is $X\subseteq V(G)$ with $|X|\leq f(d)$ such that $G-N_G[X]$ does not contain $H$ as an induced minor. 
As evidence for this conjecture, they showed that it holds when $H$ is a cycle or a `long theta' (see~\cite{CMPRR26} for the precise definition).   
Note that if this conjecture is true for some planar graph $H$, then in particular $H$ has the independence Erd\H{o}s--P\'osa property in $K_{1,t}$-free graphs. 
Thus, cycles and long thetas  have the independence Erd\H{o}s--P\'osa property in $K_{1,t}$-free graphs.

\subsection{Overview of Proof}

We now present a high-level overview of our proof. We work with tree independence number as much as possible so that our machinery may be used to tackle \cref{ConjTreeAlpha}.

The central notion to our results is what we call \emph{$\alpha$-Menger boundedness}. Let $G$ be a graph and $A,B\subseteq V(G)$. A set $X\subseteq V(G)$ is an \defn{$(A,B)$-separator} if $G-X$ does not contain an $(A,B)$-path. We define \defn{$\asep(G,A,B)$} to be the minimum integer $c$ such that $G$ contains an $(A,B)$-separator $X$ where $\alpha(X)=c$. We say that a graph class $\GG$ is \defn{$\alpha$-Menger bounded} if there exists a function $f_{\alpha}$ such that, for every graph $G\in \GG$, for every $k\in \NN$, and for every pair of induced paths $P_1,P_2$ in $G$, the following property holds: If $\asep(G,N_G[P_1],N_G[P_2])\geq f_{\alpha}(k)$ then $G$ contains $k$ pairwise anti-complete $(N_G[P_1],N_G[P_2])$-paths. We call $f_{\alpha}$ an \defn{$\alpha$-Menger binding function} for $\GG$. 

The first stage of our proof reduces \cref{ConjTreeAlpha} to showing $\alpha$-Menger boundedness. This approach for tackling \cref{ConjTreeAlpha} was suggested by \citet{CW2025Ladder}.

\begin{thm}\label{DualityMengerTreeAlpha}
    For every planar graph $H$ and $t\in \NN$, if the class of $K_{1,t}$-free graphs $\GG$ with no $H$ induced minor is $\alpha$-Menger bounded, then $\GG$ has bounded tree independence number.
\end{thm}

\begin{proof}[Proof Sketch]
     For contradiction, suppose there is a graph $G\in \GG$ with large tree independence number. Then $G$ contains a bramble of large `$\alpha$-order'. Using induced Menger, we construct an `induced grid-like minor' of large order. This grid-like minor yields a bounded-degree induced subgraph $G'$ with large treewidth. By applying \cref{KorhonenInducedGrid} on $G'$, we conclude that $G$ contains $H$ as an induced minor, a contradiction.
\end{proof}

See \cref{SectionMengerTreeAlpha,SecIndGridMinor} for the full details of the proof of \cref{DualityMengerTreeAlpha}.

Next, we show that $K_{1,t}$-free graphs excluding a fixed apex-forest induced minor is $\alpha$-Menger bounded.

\begin{thm}\label{ApexForestAlphaMengerBounded}
    For every apex-forest $H$ and $t\in \NN$, the class of $K_{1,t}$-free graphs $\GG$ with no $H$ induced minor is $\alpha$-Menger bounded.
\end{thm}

See \cref{MengerForest} for the proof of \cref{ApexForestAlphaMengerBounded}. Together with \cref{DualityMengerTreeAlpha}, this immediately implies \cref{MainThmApexForest}. In the case when $H$ is a forest, we further employ an Erd\H{o}s--P\'osa style lemma to convert the tree decomposition into a path decomposition with bounded independence number, thereby proving \cref{MainThmForest}.  

%%%%%%%%%%%%%%%%%%%%%%%%%%%%%%%%%%%%%%%%%%%%%%%%%%%%%

\section{Preliminaries}\label{SectionPrelim}

Let $\NN=\{1,2,\dots\}$ and $\NN_0=\NN\cup \{0\}$. For $n\in \NN$, let $[n]:=\{1,2,\dots,n\}$.

Let $G$ be a graph and $X,Y\subseteq V(G)$. For $v \in V(G)$, let $N_G(v):=\{u\in V(G):uv \in E(G)\}$ and $N_G[v]:=N_G(v)\cup \{v\}$. 
For $X\subseteq V(G)$, let $N_G(X):=(\bigcup_{v\in X} N_G(v))\setminus X$ and $N_G[X]:=N_G(X)\cup X$. 
For $X\subseteq V(G)$, we write $G[X]$ for the subgraph of $G$ induced on $X$. We write $G-X$ for the graph $G[V(G)\setminus X]$. We say that $X$ and $Y$ are \defn{anti-complete} if $X\cap Y=\emptyset$ and there is no edge in $G$ with one end in $X$, the other in $Y$. An \defn{$(X, Y)$-path} is a path $P=(v_1,\ldots,v_k)$ with $V(P)\cap X = \{v_1\}$ and $V(P)\cap Y = \{v_k\}$.
An $(X,Y)$-path is \defn{geodesic} if it is a shortest $(X,Y)$-path in $G$.

Throughout this paper, whenever there is no risk of confusion, we identify a vertex set $X$ with the induced subgraph $G[X]$, and a subgraph $H$ with its vertex set $V(H)$. For example, $\alpha(X)=\alpha(G[X])$ and $N_G[H]=N_G[V(H)]$.

A \defn{minor model} of a graph $H$ in a graph $G$ is a collection $\mathcal{W}=(W_v \colon v \in V(H))$ of pairwise vertex-disjoint connected subgraphs in $G$ such that $W_u$ and $W_v$ are adjacent whenever $uv\in E(H)$. Define $V(\mathcal{W})=\bigcup(W_v\colon v\in V(H))$. If $W_x$ and $W_y$ are anti-complete whenever $xy\not \in E(H)$, then $\mathcal{W}$ is an \defn{induced minor model}. Each $W_u$ is called a \defn{branch set} of the model. It is folklore that $H$ is an (induced) minor of $G$ if and only if $G$ contains an (induced) minor model of $H$.

Given a tree $T$ and a node $r$ of $T$, we denote by $(T, r)$ the tree rooted at node $r$. 
Let $(T,r)$ be a rooted tree. A node $x\in V(T)$ is a \defn{descendant} of a node $z\in V(T)$ if the (unique) $(x,r)$-path in $T$ contains $z$. If $x\neq z$, then $x$ is a \defn{strict descendant} of $z$.

Let $G$ be a graph. An \defn{independent set (clique)} in $G$ is a set of pairwise non-adjacent (adjacent) vertices. A set $D\subseteq V(G)$ is \defn{dominating} if every vertex in $V(G)\setminus D$ is adjacent to a vertex in $D$. The \defn{independence number} $\alpha(G)$ (\defn{clique number} $\omega(G)$) of $G$ is the maximum size of an independent set (clique) in $G$. The \defn{domination number} $\gamma(G)$ of $G$ is the minimum size of a dominating set in $G$. Since every inclusion-wise maximal independent set is a dominating set, we have $\gamma(G)\leq \alpha(G)$. We frequently make use of the fact that $G$ being $K_{1,t}$-free implies $\alpha(N_G[v])\leq t$ for all $v\in V(G)$.

\citet{DALLARD2024indep} observed that tree independence number is closed under taking induced minors. It is straightforward to see that the same holds for path independence number.

\begin{prop}[\cite{DALLARD2024indep}]\label{IndependentMinor}
    For every graph $G$, for every induced minor $G'$ of $G$, we have $\atw(G')\leq \atw(G)$ and $\apw(G')\leq \apw(G)$.
\end{prop}

We frequently use Ramsey's Theorem~\cite{ramsey1930problem}.

\begin{thm}[\cite{ramsey1930problem}]\label{RamseyTheorem}
    There exists a function $R$ such that, for all $r,s\in \NN$, every graph $G$ on $R(r,s)$ vertices contains an independent set of size $r$ or a clique of size $s$.
\end{thm}

\section{$\alpha$-Menger $\Rightarrow$ Tree Independence Number}\label{SectionMengerTreeAlpha}

In this section, we prove that for every $\alpha$-Menger bounded graph class, every $K_{1,t}$-free graph in the class with large tree independence number contains a large grid as an induced minor. For $k\in \NN$, the \defn{$(k\times k)$-grid} is the graph with vertex set $\{(x,y)\colon x,y\in [k]\}$ where $(x_1,y_1)$ and $(x_2,y_2)$ are adjacent whenever $|x_1-x_2|+|y_1-y_2|=1$.

\begin{thm}\label{AlphaMengerTreeInd}
     Let $\GG$ be an $\alpha$-Menger bounded graph class with $\alpha$-Menger binding function $f_{\alpha}$. There exists a function $f$ (depending on $f_{\alpha}$) such that for all $k,t\in \NN$, every $K_{1,t}$-free graph $G\in \GG$ with $\atw(G)\geq f(k,t)$ contains the $(k\times k)$-grid as an induced minor.
\end{thm}

Since grids contain all planar graphs as induced minors, this establishes \cref{DualityMengerTreeAlpha}.

Our main tool is an induced version of the grid-like minors of Reed and Wood \cite{reed2012polynomial}. Let $G$ be a graph. We say that two sets $X,Y \subseteq V(G)$ \defn{touch} if $X \cap Y \neq \emptyset$ or there is an edge of $G$ between $X$ and $Y$. For a collection $\HH$ of subgraphs of $G$, the \defn{intersection (touching) graph} $J$ of $\HH$ is the graph with vertex set $\HH$ where two vertices of $J$ are adjacent if and only if their corresponding subgraphs intersect (touch). A \defn{grid-like minor of order $\ell$} is a collection $\mathcal{H}$ of paths in $G$ whose intersection graph is bipartite and contains $K_{\ell}$ as a minor. In particular, if $G$ contains a $(n \times n)$-grid, then letting $\mathcal{H}$ be the collection of vertical and horizontal paths yields a grid-like minor of order $n+1$. Thus grid-like minors are a more general structure than grids which certifies large treewidth.

For our purposes, we require an induced variant. We define an \defn{induced grid-like minor of order $\ell$} of a graph $G$ to be a collection $\mathcal{H}$ of \emph{induced} paths in $G$ whose \emph{touching} graph is bipartite and contains $K_{\ell}$ as a minor. 

We use brambles to show that induced grid-like minors have large treewidth. For a graph $G$, a \defn{bramble} $\mathcal{B}$ in $G$ is a set of pairwise touching connected subgraphs. A set $S \subseteq V(G)$ is a \defn{hitting set} of $\mathcal{B}$ if $S$ intersects every element of $\mathcal{B}$. The \defn{order} of $\mathcal{B}$ is the minimum size of a hitting set.

\begin{thm}[\cite{seymour1993graph}]\label{TreewidthDuality}
    For every $k\in \NN$, a graph $G$ has treewidth at least $k$ if and only if $G$ has a bramble of order at least $k+1$.
\end{thm}

Using brambles, we show that in the $K_{1,t}$-free setting, induced grid-like minors correspond to induced subgraphs with bounded maximum degree and large treewidth.

\begin{lem}\label{GridLikeMinorDegree}
    Let $\ell,t\in \NN$ and let $G$ be a $K_{1,t}$-free graph that contains an induced grid-like minor of order $\ell$. Then $G$ contains an induced subgraph of maximum degree at most $2t$ and treewidth at least $\ceil{\ell/2}-1$.
\end{lem}

\begin{proof}
    Let $\mathcal{H}$ be an induced grid-like minor of $G$ with bipartite touching graph $J$ containing a $K_{\ell}$ minor. By taking an induced subgraph of $G$ if necessary, we may assume that every vertex in $G$ belongs to some path in $\mathcal{H}$.

    We first show that $G$ has large treewidth. Let $(B_x\colon x\in V(K_{\ell}))$ be a minor model of $K_{\ell}$ in $J$. For each $x\in V(K_{\ell})$, define $C_x:=G[\bigcup(V(P)\colon P\in B_x)]$. Since $B_x$ is connected, $C_x$ is also connected. Since $B_x$ and $B_y$ are adjacent in $J$, $C_x$ and $C_y$ are adjacent in $G$. Hence, $(C_x\colon x\in V(K_{\ell}))$ is a bramble in $G$. As $J$ is bipartite, every vertex $v\in V(G)$ is contained in at most two of the subgraphs $C_x$. Consequently, every hitting set of the bramble has size at least $\ceil{\ell/2}$. By \cref{TreewidthDuality}, $\tw(G) \geq \ceil{\ell/2}-1$.

    We now bound the maximum degree of $G$. Let $(X,Y)$ be a bipartition of $J$. Define $\widetilde{X}:=\bigcup (V(P) \colon P\in X)$ and $\widetilde{Y}:=\bigcup (V(P) \colon P\in Y)$. Then $V(G)= \widetilde{X} \cup \widetilde{Y}$. Fix $v\in \widetilde{X}$ and let $P_v$ denote the unique path in $X$ containing $v$. Since $J$ is a touching graph and $P_v$ is induced, we have $|N_G(v)\cap \widetilde{X}|=|N_G(v)\cap V(P_v)|\leq 2$. Since $G[\widetilde{Y}]$ is a collection of anti-complete induced paths and $\alpha(N_G(v)) \leq t-1$, we have $|N_G(v)\cap \widetilde{Y}|\leq 2(t-1)$. Therefore, $\deg_G(v)\leq 2t$. By symmetry, the same holds for all vertices in $\widetilde{Y}$, as required.
\end{proof}

\cref{GridLikeMinorDegree,KorhonenInducedGrid} gives the following.

\begin{lem}\label{GridtoGridlike}
    There exists a function $f$ such that, for all $k,t\in \NN$, every $K_{1,t}$-free graph with no $(k\times k)$-grid induced minor has no induced grid-like minor of order $f(k,t)$.
\end{lem}

The next theorem is an induced version of the main theorem of \citet{reed2012polynomial}. We postpone its proof to the next section.

\begin{restatable}{thm}{TreeAlphaGridLikeMinor}\label{TreeAlphaGridLikeMinor}
	Let $\GG$ be an $\alpha$-Menger bounded class of graphs with $\alpha$-Menger binding function $f_{\alpha}$. There exists a function $f$ (which depends on $f_{\alpha}$) such that, for all $t,\ell\in \NN$, every $K_{1,t}$-free graph $G\in \GG$ with no induced grid-like minor of order $\ell$ has $\atw(G) < f(t,\ell)$.
\end{restatable}

\cref{AlphaMengerTreeInd} now follows from \cref{GridtoGridlike,TreeAlphaGridLikeMinor}.

\section{Induced Grid-Like Minors}\label{SecIndGridMinor}
This section is devoted to proving \cref{TreeAlphaGridLikeMinor}. Our argument adapts the main ideas of \citet{reed2012polynomial} to the induced setting. 

We begin by introducing strong brambles. For a graph $G$, a bramble $\mathcal{B}$ is \defn{strong} if its elements pairwise intersect. The \defn{$\alpha$-order $\alpha(\mathcal{B})$} of $\mathcal{B}$ is the minimum independence number of a hitting set of $\mathcal{B}$. Similar to treewidth, strong brambles are a dual object to tree independence number.

\begin{thm}[\cite{CHMW2025wheel}]\label{BrambleDual}
    For every graph $G$ and $k\in \NN$:
    \begin{enumerate}
        \item If $G$ contains a strong bramble $\BB$ with $\alpha(\BB)\geq k$, then $\atw(G)\geq k$.
        \item If $\atw(G)\geq 4k-2$, then $G$ contains a strong bramble $\BB$ with $\alpha(\BB)\geq k$.
    \end{enumerate}
\end{thm}

The following result relates strong brambles to induced paths.

\begin{lem}[\cite{CHMW2025wheel}]\label{PathBramble}
    For every strong bramble $\mathcal{B}$ in a graph $G$, there exists an induced path $P$ in $G$ such that $N_G[P]$ intersects every element of $\mathcal{B}$.
\end{lem}

The next lemma breaks the path given by the above lemma into many pairwise anti-complete paths with high $\alpha$-separations between their neighbourhoods. Note that \citet{CW2025Ladder} proved the $\ell=2$ case of the following lemma.

\begin{lem}\label{AntiPathsBramble}
    For all $k,\ell,t\in \NN$ with $k> 2t$, every $K_{1,t}$-free graph $G$ with $\atw(G)\geq 4\ell(k+2t)-2$ contains $\ell$ pairwise anti-complete induced paths $P_1,\dots,P_{\ell}$ such that, for all distinct $i,j\in [\ell]$, we have $\asep(G,N_G[P_i],N_G[P_j])\geq k$.
\end{lem}

\begin{proof}
    Let $k, \ell, t$, and $G$ be as in the lemma statement. 
    By \cref{BrambleDual}, $G$ contains a strong bramble $\BB$ with $\alpha(\BB) \geq \ell(k+2t)$. Let $P=(v_1,\dots,v_n)$ be the induced path guaranteed by \cref{PathBramble}. For $a,b\in [n]$ with $a\leq b$, let $P_{a,b}$ denote the $(v_a,v_b)$-subpath of $P$, and define $\mathcal{B}_{a,b}:=\{B\in \mathcal{B}\colon V(B)\cap N_G[P_{a,b}]\neq \emptyset\}$. Observe that each $\mathcal{B}_{a,b}$ is a strong bramble in $G$. 

    Choose $a_1\in \{3,\dots,n\}$ minimal such that $\alpha(\mathcal{B}_{1,a_1-2})\geq k$. 
    Since $\alpha(\mathcal{B}_{1,2})\leq \alpha(N_G[\{v_1,v_2\}])\leq 2t<k$, we have $a_1\geq 4$. By minimality, $\alpha(\mathcal{B}_{1,a_1-3})<k$, so there exists a hitting set $S'\subseteq V(G)$ of $\mathcal{B}_{1,a_1-3}$ with $\alpha(S')<k$. Since $S'\cup N_G[\{v_{a_1-1},v_{a_1-2}\}]$ is a hitting set of $\mathcal{B}_{1,a_1-1}$, we have $\alpha(\mathcal{B}_{1,a_1-1})<k+2t$ (by $K_{1,t}$-freeness). Consequently, 
    $$\alpha(\mathcal{B}_{a_1,n})\geq \alpha(\mathcal{B}_{1,n})-\alpha(\mathcal{B}_{1,a_1-1})> \ell(k+2t)-(k+2t)\geq (\ell-1)(k+2t).$$
    
    Repeating this argument, we obtain a sequence of integers $a_0=1,a_1,\dots,a_{\ell}\leq n$ with $a_{i-1}\leq a_i-2$ for all $i\in [\ell]$, such that, $\alpha(\mathcal{B}_{a_{i-1},a_i-2})\geq k$ for each $i\in [\ell]$. 
    For each $i\in [\ell]$, set $P_i:=P_{a_{i-1},a_i-2}$. Since $P$ is induced, the paths $P_1,\dots,P_{\ell}$ are pairwise anti-complete.
    
    Suppose for contradiction that some distinct $i,j\in [\ell]$ admit an $(N_G[P_i],N_G[P_j])$-separator $S\subseteq V(G)$ with $\alpha(S)\leq k-1$. Since $\alpha(\mathcal{B}_{a_{i-1},a_{i}-2})\geq k$ and $\alpha(\mathcal{B}_{a_{j-1},a_{j}-2})\geq k$, there exist $B_i\in \mathcal{B}_{a_{i-1},a_i-2}$ and $B_j\in \mathcal{B}_{a_{j-1},a_j-2}$ such that $V(B_i)\cap S=V(B_j)\cap S=\emptyset$. Because $B_i$ and $B_j$ intersect, $G[B_i \cup B_j]$ contains an $(N_G[P_i],N_G[P_j])$-path which avoids $S$, contradicting the assumption on $S$. Thus, $\asep(G,N_G[P_i],N_G[P_j])\geq k$, as required.
\end{proof}

We need the following independent transversal lemma of \citet{reed2012polynomial}.

\begin{lem}[\cite{reed2012polynomial}]\label{IndTransversal}
    Let $r,d\in \NN$ and $V_1,\dots, V_r$ be the colour classes in an $r$-colouring of a graph $H$. Suppose that $|V_i|\geq 4e(r-1)d$ for all $i\in [r]$, and $H[V_i\cup V_j]$ is $d$-degenerate for all distinct $i,j\in [r]$. Then there exists an independent set $\{x_1,\dots,x_r\}$ of $H$ such that $x_i\in V_i$ for each $i\in [r]$.
\end{lem}

We also need the following extremal bound for $K_{\ell}$ minors.

\begin{thm}[\cite{kostochka1984average,thomason1984contraction,thomason2001extremal}]\label{ExtremalMinor}
    There exists a function $d(\ell)\in O(\ell \sqrt{\log(\ell)})$ such that, for every $\ell\in \NN$, every graph with no $K_{\ell}$ minor is $d(\ell)$-degenerate.
\end{thm}

We are now ready to construct our induced grid-like minors.

\begin{lem}\label{FindingGLMinor}
    Let $\ell\in \NN$ and $k:=\ceil{4e {\ell\choose 2} d(\ell)}$ where $d(\ell)$ is from \cref{ExtremalMinor}. Suppose a graph $G$ contains $\ell$ pairwise anti-complete induced paths $P_1,\dots,P_{\ell}$ such that, for all $i,j\in [\ell]$ with $i<j$, there is a collection $\mathcal{Q}_{i,j}$ of $k$ pairwise anti-complete induced $(N_G[P_i],N_G[P_j])$-paths in $G$. Then $G$ contains an induced grid-like minor of order $\ell$.
\end{lem}

\begin{proof}
    For all distinct unordered pairs $\{i,j\},\{a,b\}\subseteq [\ell]$ with $i\neq j$, and $a\neq b$, let $H_{i,j,a,b}$ denote the touching graph of $\mathcal{Q}_{i,j}\cup \mathcal{Q}_{a,b}$. Each $\mathcal{Q}_{i,j}$ is a collection of pairwise anti-complete paths, so $H_{i,j,a,b}$ is bipartite. If any $H_{i,j,a,b}$ contains $K_{\ell}$ as a minor, then $\mathcal{Q}_{i,j}\cup \mathcal{Q}_{a,b}$ is an induced grid-like minor of order $\ell$ in $G$, and we are done. Hence, we may assume that no $H_{i,j,a,b}$ contains a $K_{\ell}$ minor. By \cref{ExtremalMinor}, each $H_{i,j,a,b}$ is $d(\ell)$-degenerate.

    Let $H$ be the touching graph of $\bigcup(\mathcal{Q}_{i,j}\colon 1\leq i<j\leq \ell)$. Then $H$ admits a proper colouring with $m={\ell \choose 2}$ colour classes $V_1,\dots, V_m$ where each $V_p$ corresponds to one of the collections $\mathcal{Q}_{i,j}$. For all distinct $i,j\in [m]$, the bipartite subgraph $H[V_i\cup V_j]$ is $d(\ell)$-degenerate.  By \cref{IndTransversal}, $G$ contains a collection of anti-complete induced paths $Q_1,\dots,Q_m$ where $Q_p\in V_p$ for each $p\in [m]$.

    Consider the collection of induced paths $\mathcal{P}:=\{P_i\colon i\in [\ell]\}\cup \{Q_j\colon j\in [m]\}$. Let $J$ denote the touching graph of $\mathcal{P}$. Since $\{P_i\colon i\in [\ell]\}$ and $\{Q_j\colon j\in [m]\}$ are collections of anti-complete paths, $J$ is bipartite. Moreover, $J$ contains the $1$-subdivision of $K_{\ell}$ as a subgraph; indeed, for all $i,j\in [\ell]$ with $i<j$, some path in $\{Q_p\colon p\in [m]\}$ is an $(N_G[P_i],N_G[P_j])$-path. Therefore, $\mathcal{P}$ is an induced grid-like minor of order $\ell$ in $G$.
\end{proof}

Finally, we can prove \cref{TreeAlphaGridLikeMinor}.

\TreeAlphaGridLikeMinor*

\begin{proof}
    Set $k:=\max\{f_{\alpha}(\ceil{4e {\ell\choose 2} d(\ell)}),2t+1\}$ where $d(\ell)$ is from \cref{ExtremalMinor} and define $f(t,\ell):=4\ell(k+2t)-2$. Then $k>2t$. For the sake of contradiction, suppose that $\atw(G)\geq f(t,\ell)$. By \cref{AntiPathsBramble}, $G$ contains $\ell$ pairwise anti-complete induced paths $P_1,\dots,P_{\ell}$ such that, for all distinct $i,j\in [\ell]$, we have $\asep(G,N_G[P_i],N_G[P_j])\geq k\geq f_{\alpha}(\ceil{4e {\ell\choose 2} d(\ell)})$. By the definition of $\alpha$-Menger boundedness, it follows that, for all distinct $i,j\in [\ell]$, $G$ contains a collection $\mathcal{Q}_{i,j}$ of $\ceil{4e {\ell\choose 2} d(\ell)}$ pairwise anti-complete $(N_G[P_i],N_G[P_j])$-paths. By taking these paths to be vertex-minimal, we may assume that they are induced. By \cref{FindingGLMinor}, $G$ contains an induced grid-like minor of order $\ell$, a contradiction.
\end{proof}

\section{$\alpha$-Menger and Induced Apex-Forest Minor}\label{MengerForest}

In this section, we prove the following theorem.

\begin{thm}\label{AlphaMengerForestMinor}
    For every apex-forest $H$ and $t\in \NN$, the class of $K_{1,t}$-free graphs that exclude $H$ as an induced minor is $\alpha$-Menger bounded.
\end{thm}

\cref{AlphaMengerForestMinor} is one of the main technical results of this paper. We begin with a high-level overview of its proof. Let $G$ be a $K_{1,t}$-free graph with no $H$-induced-minor, and let $Q_1$ and $Q_2$ be two paths in $G$. Suppose that $\asep(G,N_G[Q_1],N_G[Q_2])$ is large. Our goal is to show that $G$ contains many pairwise anti-complete $(N_G[Q_1],N_G[Q_2])$-paths.

The proof has four main ingredients. First, in \cref{SecAlphaSeparators}, we establish basic lemmas controlling how the independence number of separators changes with various modifications. Second, since $G$ is $K_{1,t}$-free and excludes a fixed planar graph $H$ as an induced minor, we show that $G$ also excludes a sufficiently large complete bipartite induced minor (see \cref{ExcludedBipartite}). We then prove some technical lemmas in \cref{SecBipartite} which force such a large complete bipartite induced minor whenever there is too much interaction between paths joining $N_G[Q_1]$ and $N_G[Q_2]$. These lemmas are then used in \cref{SecSinglePath} to extract a single $(N_G[Q_1],N_G[Q_2])$-path $P$ such that the separation between $N_G[Q_1]$ and $N_G[Q_2]$ remains large after deleting $N_G[P]$. Finally, in \cref{SecProofAlphaMenger}, we iterate this extraction step. Each iteration produces one more path while preserving enough separation to continue. This yields the required collection of pairwise anti-complete $(N_G[Q_1],N_G[Q_2])$-paths, completing the proof of \cref{AlphaMengerForestMinor}.

\subsection{Properties about $\alpha$-separators}\label{SecAlphaSeparators}

Before proving \cref{AlphaMengerForestMinor}, we first establish some useful properties about separators.

\begin{obs}\label{BasicLemSep3}
    For all graphs $G$ and $G'$ where $G'\subseteq G$, sets $A,B,A'\subseteq V(G)$ where $A'\subseteq A$, we have $ \max\{\asep(G',A,B),\asep(G,A',B)\}\leq \asep(G,A,B)$.
\end{obs}

\begin{proof}
    Let $X$ be an $(A,B)$-separator in $G$. Then $X$ is an $(A',B)$-separator in $G$ and $X\cap V(G')$ is an $(A,B)$-separator in $G'$, implying the claim.
\end{proof}

\begin{obs}\label{BasicLemSep2}
    For every $t\in \NN$, all sets $A,B,C\subseteq V(G)$, if $\alpha(C)\leq t$ then $\min\{\asep(G-C,A,B),\asep(G,A-C,B)\}\geq \asep(G,A,B)-t$.  
\end{obs}

\begin{proof}
    Let $X$ be an $(A,B)$-separator in $G-C$. Then $X\cup C$ is an $(A,B)$-separator in $G$. Since $\alpha(C)\leq t$, we obtain $\alpha(X) + t\geq \asep(G,A,B)$, implying $\asep(G-C,A,B)\geq\asep(G,A,B)-t$. Since $\asep(G,A-C,B)\geq \asep(G-C,A,B)$, the claim follows.
\end{proof}

\begin{lem}\label{BasicLemSep1}
    For every $k\in\NN$, for every graph $G$ and all sets $A,B,C\subseteq V(G)$ with $A\cap B=\emptyset$, if $\asep(G,A,C)-\asep(G-B,A,C)\geq k$ where $k\leq \asep(G,A,C)$, then $\asep(G-A,B,C)\geq k$.  
\end{lem}

\begin{proof}
    Since $\asep(G-B,A,C)\leq \asep(G,A,C)-k$, there exists a set $X_1\subseteq V(G)\setminus B$ with $\alpha(X_1)\leq \asep(G,A,C)-k$ such that $G-B-X_1$ contains no $(A,C)$-path. Suppose for contradiction that $\asep(G-A,B,C)< k$. Then there exists a set $X_2\subseteq V(G)\setminus A$ with $\alpha(X_2)<k$, such that every $(B,C)$-path in $G-A$ intersects $X_2$. Let $P$ be an $(A,C)$-path in $G$. Then only the first vertex in $P$ belongs to $A$. If $V(P)\cap B=\emptyset$, then $P$ is an $(A,C)$-path in $G-B$, hence it contains a vertex from $X_1$. Otherwise, let $v$ be the final vertex of $P$ that belongs to $B$. Then $v\not\in A$ since $A\cap B=\emptyset$. As such, the subpath of $P$ from $v$ to the terminal vertex in $C$ is a $(B,C)$-path in $G-A$, so it contains a vertex from $X_2$. Thus, every $(A,C)$-path in $G$ contains a vertex from $X_1\cup X_2$, so $X_1\cup X_2$ is an $(A,C)$-separator. But $\alpha(X_1\cup X_2)\leq \alpha(X_1)+\alpha (X_2)<\asep(G,A,C)$, a contradiction.
\end{proof}

\subsection{Finding complete bipartite graphs as induced minors}\label{SecBipartite}

We will use the following result of \citet{CHS2024Bipartite}, which describes the unavoidable induced structures in graphs that contain a large complete bipartite graph as an induced minor while excluding a fixed planar graph as an induced minor.  An \defn{$(s,\ell)$-constellation} is an induced minor model $(\{v_1\},\dots,\{v_s\},P_1,\dots,P_{\ell})$ of $K_{s,\ell}$ in which the branch sets on one side of the bipartition are single vertices, and the branch sets on the other side are induced paths. 

\begin{thm}[\citet{CHS2024Bipartite}]\label{MainCHS}
    There is a function $f_B'$ such that, for all $n,\ell,s\in \NN$, for any $n$-vertex planar graph $H$, if $c:=f_B'(n,\ell,s)$, then every $H$-induced-minor-free graph that contains $K_{c,c}$ as an induced minor contains an $(s,\ell)$-constellation.
\end{thm}
Observe that a $(1,t)$-constellation contains $K_{1,t}$ as an induced subgraph. Hence, by setting $f_B(n,t):=f_B'(n,t,1)$, \cref{MainCHS} immediately gives the following. 

\begin{lem}\label{ExcludedBipartite}
    There is a function $f_B$ such that, for all $n,t\in \NN$, for any $n$-vertex planar graph $H$, if $c:=f_B(n,t)$, then every $K_{1,t}$-free $H$-induced-minor-free graph does not contain $K_{c,c}$ as an induced minor.
\end{lem}

Thus, for proving \cref{AlphaMengerForestMinor}, we may assume that we also exclude large complete bipartite graphs as induced minors. We now prove some technical lemmas to find these induced minors.

The classic K\H{o}v\'ari--S\'os--Tur\'an Theorem~\cite{kovari1954zarankiewicz} states that every $n$-vertex graph $G$ containing no $K_{m,m}$ subgraph has $O(n^{2-1/m})$ edges. 
The following lemma is a straightforward consequence of this result. 

\begin{lem}\label{lemBipartite}
    There is a function $f_K$ such that, for all $\epsilon_1,\epsilon_2>0$ and $m\in \NN$, if $G=(A,B)$ is a bipartite graph containing no $K_{m,m}$ subgraph with $|B|\geq |A|\geq \epsilon_1|B|-1$ and every vertex in $A$ has degree at least $\epsilon_2|B|$, then $|B| < f_K(\epsilon_1,\epsilon_2,m)$.
\end{lem}

\begin{proof}
    Set $a:=|A|$ and $b:=|B|$. Since every vertex in $A$ has degree at least $\epsilon_2 b$, we have $|E(G)|\geq a\epsilon_2 b \geq \epsilon_1\epsilon_2 b^2- \epsilon_2 b$. On the other hand, $G$ is $K_{m,m}$-free with at most $a+b\leq 2b$ vertices. Hence, by the K\H{o}v\'ari--S\'os--Tur\'an Theorem~\cite{kovari1954zarankiewicz}, we have $|E(G)|\in O(b^{2-1/m})$. Combining the two bounds, gives $\epsilon_1\epsilon_2 b^2- \epsilon_2 b\in O(b^{2-1/m})$.
    Since the left-hand side grows strictly faster than the right-hand side as $b\to \infty$, it follows that $b$ is bounded above by a function of $\epsilon_1,\epsilon_2,$ and $m$, as required.
\end{proof}

For a graph $G$ and $n\in \NN$, we define an \defn{$n$-array} to be a tuple $(P_1,\dots,P_{n},Q_1,\dots,Q_n, D_1,\dots,D_n)$ satisfying:
\begin{enumerate}[label={$(C{\arabic*})$}]
    \item\label{C1} $P_1,\dots, P_{n},Q_1,\dots,Q_n$ are pairwise anti-complete connected subgraphs in $G$; 
    \item\label{C2} $D_i\subseteq N_G[P_i]$ for every $i\in [n]$;
    \item\label{C3} for all $i,j\in [n]$, the set $D_i$ has a neighbour in $Q_j$; and
    \item\label{C4} for all $1\leq j<i\leq n$, the set $D_j$ is anti-complete to $P_i\cup D_i$. 
\end{enumerate}

\begin{lem}\label{array}
    There exists a function $f_C$ such that, for all $r,s,t\in \NN$, every $K_{1,t}$-free graph $G$ containing an $n$-array $(P_1,\dots,P_{n},Q_1,\dots,Q_n,D_1,\dots,D_n)$ where $n=f_C(r,s,t)$, contains an induced minor model $(X_1,\dots,X_r,Y_1,\dots,Y_s)$ of $K_{r,s}$ such that each $X_i$ is contained in $G[P_a\cup D_a]$ for some $a\in [n]$, and each $Y_j$ is equal to $Q_{b}$ for some $b\in [n]$.
\end{lem}

\begin{proof}
    We make no attempt to optimise the function $f_C$. We proceed by induction on $r$. 
    
    For $r=1$, let $f_C(1,s,t):=s$. Given an $s$-array, set $X_1:=G[P_1\cup D_1]$    and $Y_j:=Q_{j}$ for each $j\in [s]$. Then $X_1$ is connected since $P_1$ is connected and $D_1\subseteq N_G[P_1]$. Moreover, $X_1$ is adjacent to each $Y_j$ since $D_1$ has a neighbour in each $Q_{j}$. Finally, $Y_1,\dots,Y_s$ are pairwise anti-complete connected subgraphs, giving the desired induced minor model of $K_{1,s}$. 
    
    Now assume $r\geq 2$, and that the statement holds for $r-1$. Let $m:=f_C(r-1,s,t)$ and set 
    $$f_C(r,s,t):=f_K(\frac{1}{4t^2+1},\frac{1}{2t},m)$$ 
    where $f_K$ is the function from \cref{lemBipartite}. Let $n:=f_C(r,s,t)$.

    Let $(P_1,\dots,P_{n},Q_1,\dots,Q_n,D_1,\dots, D_n)$ be an $n$-array in $G$. By deleting excess vertices, we may assume that each $D_i$ is minimal subject to having a neighbour in each $Q_j$. So $|D_i|\leq n$. Moreover, for each $v\in D_i$, there exists $j\in [n]$ such that $N_G(Q_j)\cap D_i=\{v\}$. Since $Q_{1},\dots, Q_{n}$ are pairwise anti-complete and $G$ is $K_{1,t}$-free, each vertex of $D_i$ has neighbours in at most $t-1$ of these paths. Hence $|D_i|\geq \frac{n}{t}$.
    
    Define an auxiliary digraph $\overrightarrow{H}$ on vertex-set $[n]$ where $\overrightarrow{ij}$ is an arc in $\overrightarrow{H}$ if $|D_i\cap N_G(P_j)|\geq \frac{n}{2t}$. We claim that $\overrightarrow{H}$ contains a large independent set. Fix $i\in [n]$. Since each vertex of $D_i$ is adjacent to at most $t-1$ of the pairwise anti-complete paths $P_1,\dots,P_n$, we have 
    $$\sum_{j=1}^n |D_i \cap N_G(P_j)|\leq (t-1)|D_i|\leq (t-1)n.$$ 
    By the construction of $\overrightarrow{H}$, it follows that the node $i\in V(\overrightarrow{H})$ has outdegree at most $\frac{(t-1)n}{n/2t}< 2t^2$. Thus, the underlying graph of $\overrightarrow{H}$ has average degree less than $4t^2$, and hence $\overrightarrow{H}$ contains an independent set $I$ of size at least $\frac{n}{4t^2+1}$. 
    
    Let $a\in I$ be the minimal element of $I$, and set $X_r:=G[P_a\cup D_a]$. Then $X_r$ is connected and is adjacent to each connected subgraph $Q_j$. Set $I':=I\setminus\{a\}$. For each $i\in I'$, define $D_i':=D_i\setminus N_G[X_r]$. By \cref{C1}, \cref{C4} and the minimality of $a$, $X_r$ is anti-complete to $P_i\cup D_i'$ for all $i\in I'$. Since $I$ is an independent set in $\overrightarrow{H}$, $|D_i\cap N_G[X_r]|< \frac{n}{2t}$ and so $|D_i'|>\frac{n}{t}-\frac{n}{2t}= \frac{n}{2t}$. By the minimality of $D_i$, it follows that $D_i'$ is adjacent to at least $\frac{n}{2t}$ of the connected subgraphs $Q_{1},\dots,Q_{n}$.

    Define a bipartite graph $H'$ with bipartition $(I',[n])$ where $ij\in E(H')$ if $D_i'$ is adjacent to $Q_j$. Then every vertex of $I'$ has degree at least $\frac{n}{2t}$ in $H'$ and $|I'|\geq \frac{n}{4t^2+1}-1$. By the choice of $n$, it follows from \cref{lemBipartite} that $H'$ contains $K_{m,m}$ as a subgraph. Thus, there exist sets $A\subseteq I'$ and $B\subseteq [n]$ with $|A|=|B|=m$, such that, for every $i\in A$ and $j\in B$, the set $D_i'$ has a neighbour in $Q_j$.

    We claim that $(P_i\colon i\in A), (Q_j\colon j\in B), (D_i'\colon i\in A)$ forms an $m$-array after re-indexing in increasing order. Indeed:
    \begin{itemize}
        \item all the chosen connected subgraphs remain pairwise anti-complete;
        \item $D_i'\subseteq D_i\subseteq N_G[P_i]$ for each $i\in A$;
        \item if $i\in A$ and $j\in B$, the set $D_i'$ is adjacent to $Q_j$; and
        \item if $i,i'\in A$ with $i'<i$, then $D_{i'}$ is anti-complete to $P_i\cup D_i$, and hence $D_{i'}'\subseteq D_{i'}$ is anti-complete to $P_i\cup D_i'$.
    \end{itemize}
    Thus it is indeed an $m$-array. Applying the induction hypothesis to this $m$-array, it follows that $G$ contains an induced minor model $(X_1,\dots,X_{r-1},Y_1,\dots,Y_s)$ of $K_{r-1,s}$ where each $X_i$ is contained in $G[P_c\cup D_c]$ for some $c\in A$, and each $Y_j$ is one of the connected subgraphs $Q_{b}$ with $b\in B$. Since $X_r$ is adjacent to every connected subgraph $Q_b$ with $b\in B$, it follows that $X_r$ is adjacent to each $Y_j$. Moreover, $X_r$ is anti-complete to $P_c\cup D_c$ for every $c\in A$, and thus $X_r$ is anti-complete to each of the branch sets $X_1,\dots,X_{r-1}$. Consequently, $(X_1,\dots,X_{r-1},X_r,Y_1,\dots,Y_s)$ is an induced minor model of $K_{r,s}$ in $G$ satisfying the condition of the lemma.
\end{proof}

\begin{lem}\label{BipartiteRamsey}
    There exists a function $f_R$ such that, for all $\ell,n\in \NN$ the following holds. Let $G$ be a $K_{n,n}$-free bipartite graph with bipartition $(A,B)$ where $|A|\geq f_R(\ell,n)$. Suppose that $B$ has a partition $B_1,\dots,B_q$, where $q\leq \ell$ and $|B_j|\geq n$ for every $j\in[q]$. Then there exists a set $A_q\subseteq A$ with $|A_q|\geq \ell$, such that, for every $j\in [q]$, there is a vertex $b_j\in B_j$ with no neighbour in $A_q$.
\end{lem}

\begin{proof}
    Let $f_R(\ell,n):=\ell n^\ell$. Set $A_0:=A$ and for each $i\in[q]$, set $k_i:=\ell n^{q-i}$. We prove by induction on $i\in \{0,1,\dots,q\}$ that there exists a set $A_i\subseteq A$ such that $|A_i|\geq k_i$ and, for each $j\in [i]$, there is a vertex $b_j\in B_j$ with no neighbour in $A_i$. 
    
    For $i=0$, this holds with $A_0:=A$. Now suppose that $i>0$ and that the claim holds for $i-1$. For the sake of contradiction, suppose that every vertex of $B_i$ has fewer than $k_i$ non-neighbours in $A_{i-1}$. By induction $|A_{i-1}|\geq \ell n^{q-i+1}=nk_i$. Let $X\subseteq B_i$ with $|X|=n$. Then  
    $$|\bigcap (N_G(v)\colon v\in X)|\geq |A_{i-1}|-n(k_i-1)\geq n.$$ 
    So the $n$ vertices of $X$ have at least $n$ common neighbours in $A_{i-1}$, contradicting $G$ being $K_{n,n}$-free. Thus there is a vertex $b_i\in B_i$ with at least $k_i$ non-neighbours in $A_{i-1}$. Define $A_i:=A_{i-1}\setminus N_G(b_i).$ Then $|A_i|\geq k_i$ and $b_i$ has no neighbour in $A_i$. Since $A_i\subseteq A_{i-1}$, it follows by induction that for every $j< i$, there is a vertex $b_j\in B_j$ with no neighbour in $A_i$. This completes the induction hypothesis. The claim then follows when $i=q$.
\end{proof}

\subsection{Extracting a Single Path}\label{SecSinglePath}
In this subsection, we will show that if a graph $G$ does not contain a forest as an induced minor and if $\asep(G,A,B)$ is sufficiently large for some sets $A,B\subseteq V(G)$, then there is an $(A,B)$-path $P$ between $A$ and $B$ such that $\asep(G-N_G[P],A,B)$ is still large. This will be the core lemma needed to prove \cref{AlphaMengerForestMinor}.

Let $(T,x)$ be a rooted tree. A vertex is a \defn{leaf} if it has no strict descendants, and a \defn{non-leaf} otherwise. For $m\in \NN_0$ and $\ell \in \NN$, we call $(m,\ell)$ a \defn{tree-pair} if there exists a non-empty tree $(T,x)$ with exactly $m$ non-leaves and $\ell$ leaves. It follows immediately that the set of tree pairs is $\{(0,1)\}\cup \{(m,\ell)\colon m,\ell\in \NN\}$. The only rooted tree with $0$ non-leaves is $(K_1,x)$. For every $m,\ell \in \NN$, one can take a rooted star with $\ell$ leaves then subdivide an edge $m-1$ times so that the tree has $m$ non-leaves.

\begin{lem}\label{TechnicalInducedMenger}
    There exists a function $f_{\mathcal{W}}$ such that, for every $m\in \NN_0$ and $c,\ell,s,t\in \NN$ where $s>2t$ and $(m,\ell)$ is a tree-pair, the following holds. Let $(T,x)$ be a rooted tree where $M$ is the set of non-leaves, $L$ is the set of leaves, $|M|=m$, and $|L|=\ell$. Let $G$ be a $K_{1,t}$-free graph and let $A,B\subseteq V(G)$. If $\asep(G,A,B)\geq f_{\mathcal{W}}(c,m,\ell,s,t)$, then one of the following holds:
    \begin{enumerate}[label={$(O{\arabic*})$}]
        \item \label{outcome_model}
        $G$ contains an induced subgraph $\widetilde{G}$ which contains an induced minor model $\mathcal{W}=(W_y\colon y\in V(T))$ of $(T,x)$ satisfying:
        \begin{enumerate}[label={$(W{\arabic*})$}]
            \item\label{W1} $\asep(\widetilde{G}-N_{\widetilde{G}}[V(\mathcal{W})],A,B)\geq s$;
            \item\label{W2} for each $y\in L$,
            $\asep(\widetilde{G}-(N_{\widetilde{G}}[V(\mathcal{W})\setminus W_y]\cup A),N_{\widetilde{G}}[W_y],B)\geq s$;
            \item\label{W3} $W_x\cap A\neq \emptyset$; and
            \item\label{W4} $V(\mathcal{W})\cap B =\emptyset$;
        \end{enumerate}
        \item\label{outcome_path}
        there exists an $(A,B)$-path $P$ in $G$ such that $\asep(G-N_G[P],A,B)\geq  s$; or
        \item\label{outcome_Knn} $G$ contains $K_{c, c}$ as an induced minor.
    \end{enumerate} 
\end{lem}

\begin{proof}
    We proceed by induction on $m$. The function $f_{\mathcal{W}}$ will be defined recursively on $m$; no attempt is made to optimise it. 

    Suppose first that $m=0$. Since $(0,1)$ is the only tree-pair with $m=0$, it follows that $\ell=1$ and $(T,x)=(K_1,x)$. We show that one may take $f_{\mathcal{W}}(c,0,1,s,t):=2s+3t$.
    Assume that $\asep(G,A,B)\geq 2s+3t$ and let $P=(v_1,v_2,\dots,v_p)$ be a geodesic $(A,B)$-path in $G$. Then $P$ is induced, only its first vertex lies in $A$, and only its last vertex lies in $B$.   
    For $1\leq a\leq b\leq p$, let $P_{(a,b)}$ denote the $(v_a,v_b)$-subpath of $P$. If $\asep(G-N_G[P],A,B)\geq s$, then outcome \cref{outcome_path} holds and we are done. So we may assume otherwise. Let $i\in \{0, 1, \dots, p-1\}$ be minimal such that $\asep(G-N_G[P_{(1,i+1)}],A,B)<  s$. Such an $i$ exists since $i=p-1$ is a candidate. Then $i\geq 2$ since $\alpha(N_G[P_{(1,3)}])\leq 3t$. Moreover, $V(P_{(1,i)})\cap B=\emptyset$. By minimality, it follows that $\asep(G-N_G[P_{(1,i)}],A,B)\geq  s$. 
    By \cref{BasicLemSep2}, we have  
    $\asep(G-N_G[P_{(1,3)}],A,B)\geq (2s+3t)-3t\geq s.$
    Applying \cref{BasicLemSep2,BasicLemSep3}, we have
    $$\asep(G-N_G[P_{(3,i+1)}],A,B)\leq \asep(G-N_G[P_{(1,i+1)}],A,B)+2t\leq s+2t.$$
    This means that $\asep(G,A,B)-\asep(G-N_G[P_{(3,i+1)}],A,B)\geq (2s+3t)-(s+2t)= s+t$. Since $P$ is geodesic and $i\geq 2$, it follows that $N_G[P_{(3,i+1)}]\cap A=\emptyset$. Applying \cref{BasicLemSep1} we obtain $\asep(G-A,N_G[P_{(3,i+1)}],B)\geq  s+t$. By applying \cref{BasicLemSep2,BasicLemSep3} once more we have
     $$\asep(G-A,N_G[P_{(1,i)}],B)\geq \asep(G-A,N_G[P_{(3,i)}],B)\geq \asep(G-A,N_G[P_{(3,i+1)}],B)-t\geq s.$$
     Set $W_x:=P_{(1,i)}$ and $\widetilde{G}:=G$. Then $\mathcal{W}=(W_x)$ is an induced minor model of $(T,x)$ with
     \begin{enumerate}[label={$(W{\arabic*})$}]
         \item  $\asep(\widetilde{G}-N_{\widetilde{G}}[\mathcal{W}],A,B)=\asep(G-N_G[P_{(1,i)}],A,B)\geq s$; 
         \item $\asep(\widetilde{G}-A,N_{\widetilde{G}}[W_x],B)=\asep({G}-A,N_G[P_{(1,i)}],B)\geq s$; 
         \item $v_1\in W_x\cap A$; and 
         \item $V(\mathcal{W})=W_x\subseteq V(P)\setminus\{v_p\}$ which is vertex disjoint from $B$.
     \end{enumerate}
     Thus we obtain outcome \cref{outcome_model}.
     
     Now assume that $m>0$, and that the statement holds for $m-1$. Fix $c,\ell,s,t$ with $s>2t$, and let $(T,x)$ be a rooted tree with $|M|=m$ and $|L|=\ell$. Let $v$ be a leaf of the tree $T[M]$, and set $L_v:=N_T(v)\cap L$. Set $n:=f_C(c,c,t)$ where $f_C$ is the function from \cref{array}. Let $\widetilde{N}:=f_R(\ell,n)$ where $f_R$ is the function from \cref{BipartiteRamsey} and set $N:=R(\widetilde{N},2n)$ where $R$ is the Ramsey function (see \cref{RamseyTheorem}). Construct a tree $T^\ast$ from $T$ as follows:
    \begin{itemize}
        \item delete every leaf in $L_v$; and
        \item for each leaf $z\in L\setminus L_v$, replace $z$ by $n$ clones $z^{(1)},\dots,z^{(n)}$, each adjacent to the unique neighbour of $z$ in $T$.
    \end{itemize}
    Since $m>0$, the root $x$ is a non-leaf of $T$, so it is an element of $V(T^{\ast})$. Let $M^\ast$ and $L^\ast$ denote the set of non-leaves and leaves of $T^\ast$ respectively. Then $v$ is a leaf of $T^\ast$, no new non-leaf is created, hence $M^\ast=M\setminus\{v\}$. In particular, $|M^\ast|=m-1$ and $|L^{\ast}|\leq n\ell$. 

    Set $\Delta:=t(n\ell+N)$, $s_N:=s+2t$, and $s_i:=2s_{i+1}+2\Delta+3t$ for each $i\in \{0,1,\dots,N-1\}$. Define $f_{\mathcal{W}}(c,m,\ell,s,t):=\max \{f_{\mathcal{W}}(c,m-1,\ell',s_0,t)\colon \ell'\in [n\ell], \text{$(m-1,\ell')$ is a tree-pair}\}$.
    
    Assume that $\asep(G,A,B)\geq f_{\mathcal{W}}(c,m,\ell,s,t)$. Apply the induction hypothesis to $(T^{\ast},x)$ with parameter $s_0$. If outcome \cref{outcome_path} or \cref{outcome_Knn} occurs, then we are immediately done. Hence we may assume that outcome \cref{outcome_model} occurs. Thus there exist an induced subgraph ${G}^{\ast}$ of $G$ and an induced minor model $\mathcal{W}^\ast=(W_y^{\ast}\colon y\in V(T^\ast))$ of $(T^\ast,x)$ in $G^{\ast}$ satisfying \cref{W1}--\cref{W4}, with $s_0$ in place of $s$.

    The next claim extracts many candidate branch sets $R_1,\dots, R_N$ for the vertices of $L_v$. In this claim, we may assume that there is no $(A,B)$-path $P$ in $G^{\ast}$ satisfying $\asep(G^{\ast}-N_G[P],A,B)\geq s$, otherwise we obtain outcome \cref{outcome_path} via \cref{BasicLemSep3}. This implies that $A\cap B=\emptyset$, otherwise we may take $P$ to be a singleton vertex in $A\cap B$. 
    
    For the following claim, set
    $\widetilde{G}_0:=G^\ast,$ $\mathcal{X}_0:=V(\mathcal{W}^\ast),$ and $R_0=\widetilde{R}_0=\widetilde{D}_0:=\emptyset$.

\begin{claim}
    There exist tuples
    $$(R_1,\widetilde{R}_1,D_1,\widetilde{D}_1,\mathcal{X}_1,\widetilde{G}_1),\dots, (R_N,\widetilde{R}_N,D_N,\widetilde{D}_N,\mathcal{X}_N,\widetilde{G}_N)$$
    such that, for every $i\in[N]$, the following properties hold:
    \begin{enumerate}[label={$(R{\arabic*})$}]
        \item\label{R1} $R_i$ is a non-empty path in $G^\ast$ that is adjacent to $W_v^\ast$;
        
        \item\label{R2} $V(R_i)\cap \bigl(N_{G^\ast}[V(\mathcal{W}^\ast)\setminus W_v^\ast]\cup W_v^{\ast}\cup A\cup B\bigr)=\emptyset;$
        
        \item\label{R3} $ \mathcal{X}_i:=V(\mathcal{W}^\ast)\cup \bigcup(V(R_j)\colon j\in[i]);$
        
        \item\label{R4} $D_i\subseteq N_{\widetilde{G}_{i-1}}(R_i)\setminus \mathcal{X}_i$  is minimal with the following property: whenever $\bigl(N_{\widetilde{G}_{i-1}}(R_i)\cap N_{\widetilde{G}_{i-1}}(X)\bigr)\setminus \mathcal{X}_i \neq\emptyset$ for some $X\in \{W_y^\ast:y\in L^\ast\setminus\{v\}\}\cup \{R_j:j<i\},$ then $N_{\widetilde{G}_{i-1}}(X)\cap D_i\neq\emptyset. $
        
        \item\label{R5} $\widetilde{D}_i:=N_{\widetilde{G}_{i-1}}[D_i]\setminus \mathcal{X}_i;$
        
        \item\label{R6}$ \widetilde{G}_i:=G^\ast-\bigcup (\widetilde{D}_j \colon j\in[i]);$
        
        \item\label{R7} $R_i$ is anti-complete in $G$ to $R_j\cup D_j$ for every $j<i$;
        
        \item\label{R8} $D_i$ is anti-complete in $G$ to $D_j$ for every $j<i$;
        
        \item\label{R9}
        $$ \widetilde{R}_i:= N_{\widetilde{G}_i}[R_i]\setminus \left( \bigcup(N_{\widetilde{G}_i}[W_y^\ast]\colon {y\in L^\ast}) \cup \bigcup(N_{\widetilde{G}_i}[R_j] \colon {j<i})\cup A \right); $$
    
        \item\label{R10}
        for every $y\in L^\ast\setminus\{v\}$,
        $$ \asep\left( \widetilde{G}_i- \left( N_{\widetilde{G}_i}[V(\mathcal{W}^\ast)\setminus W_y^\ast] \cup \bigcup(\widetilde{R}_j\colon {j\in[i]}) \cup A  \right), N_{\widetilde{G}_i}[W_y^\ast], B \right) \geq s_i; $$
    
        \item\label{R11}
        $$\asep\left(\widetilde{G}_i-\left(N_{\widetilde{G}_i}[V(\mathcal{W}^\ast)\setminus W_v^\ast]\cup N_{\widetilde{G}_i}\left[\bigcup(R_j\colon {j\in[i]})\right]\cup A\right),N_{\widetilde{G}_i}[W_v^\ast],B\right)\geq s_i;$$
    
        \item\label{R12}
        for every $j\in[i]$,
        $$ \asep\left( \widetilde{G}_i- \left( N_{\widetilde{G}_i}[V(\mathcal{W}^\ast)] \cup \bigcup(\widetilde{R}_b\colon {b\in[i]\setminus\{j\}}) \cup A \right), \widetilde{R}_j, B \right) \geq s_i; $$
    
        \item\label{R13}
        $ \asep\left( \widetilde{G}_i-N_{\widetilde{G}_i}[\mathcal{X}_i], A,B \right) \geq s_i. $
    \end{enumerate}
\end{claim}
\begin{proof}
While the statement of the claim is for $i\in [N]$, we first point that \cref{R10}--\cref{R13} also holds for $i=0$, since in this case \cref{W1} implies \cref{R13}, \cref{W2} implies \cref{R10}--\cref{R11}, and \cref{R12} is vacuous. 
We construct the tuples inductively. Let $i\in [N]$ and suppose that the tuples have already been constructed for all indices less than $i$, and that they satisfy \cref{R1}--\cref{R13}.

Set
$$H_i:=\widetilde{G}_{i-1}-\left(N_{\widetilde{G}_{i-1}}[V(\mathcal{W}^\ast)\setminus W_v^\ast]\cup N_{\widetilde{G}_{i-1}}\left[\bigcup(R_j\colon {j<i})\right]\cup A
\right).$$
By \cref{R11} at stage $i-1$, we have $ \asep(H_i,N_{H_i}[W_v^\ast],B)\geq s_{i-1}.$ Since $W_v^\ast\cap B=\emptyset$ by \cref{W4}, every path from
$N_{H_i}[W_v^\ast]$ to $B$ either starts in $N_{H_i}(W_v^\ast)$
or first meets $N_{H_i}(W_v^\ast)$ after leaving $W_v^\ast$. It follows that 
\[ \asep(H_i-V(W_v^\ast),N_{H_i}(W_v^\ast),B)=\asep(H_i,N_{H_i}[W_v^\ast],B)\geq s_{i-1}.\]
Let $Q_i=(u_1,\dots,u_p)$ be a geodesic $(N_{H_i}(W_v^\ast),B)$-path in $H_i-V(W_v^\ast)$. Such a path exists since $s_{i-1}>0$. Then $V(Q_i)\cap N_{H_i}(W_v^\ast)=\{u_1\}$ and $V(Q_i)\cap B=\{u_p\}$. Moreover, since $Q_i\subseteq H_i-V(W_v^{\ast})$, it follows from the definition of $H_i$ that
$$V(Q_i)\cap\left(N_{\widetilde{G}_{i-1}}[V(\mathcal{W}^\ast)\setminus W_v^\ast]\cup W_v^{\ast}\cup N_{\widetilde{G}_{i-1}}\left[\bigcup(R_j\colon {j<i})\right] \cup A\right) =\emptyset.$$
Thus $Q_i$ is vertex-disjoint from $V(\mathcal{W}^\ast)\cup A$ and is anti-complete to $V(\mathcal{W}^\ast\setminus W_v^{\ast})\cup \left[\bigcup(R_j\colon {j<i})\right]$. By \cref{R3}, it is vertex-disjoint to $\mathcal{X}_{i-1}$. 

We now show that $Q_i$ is anti-complete in $G$ to $D_j$ for every $j\in [i-1]$. Fix $j\in [i-1]$. By \cref{R5} and \cref{R6}, we have $N_{G}[D_j]\cap V(\widetilde{G}_{i-1}) \subseteq \mathcal{X}_j$. Since $\mathcal{X}_j\subseteq \mathcal{X}_{i-1}$, it follows that $Q_i$ is anti-complete to $D_j$.

Define
$$\widetilde{Q}_i:= N_{\widetilde{G}_{i-1}}[Q_i]\setminus \left( \bigcup(N_{\widetilde{G}_{i-1}}[W_y^\ast]\colon {y\in L^\ast}) \cup \bigcup(N_{\widetilde{G}_{i-1}}[R_j]\colon {j<i}) \cup A \right). $$
 
For the sake of contradiction, suppose that
$$\asep\left(\widetilde{G}_{i-1}-\left(N_{\widetilde{G}_{i-1}}[\mathcal{X}_{i-1}]
\cup\widetilde{Q}_i\right),A,B\right)\geq s+2t.$$
Let $J_i:=\widetilde{G}_{i-1}[V(\mathcal{W}^\ast)\cup V(Q_i)].$ Since $\mathcal{W}^\ast$ and $Q_i$ are both connected and adjacent in $\widetilde{G}_{i-1}$, it follows that $J_i$ is connected. By \cref{W3}, $V(J_i)\cap A\neq \emptyset$. Since $u_p\in V(J_i)\cap B$, $J_i$ contains an $(A,B)$-path. Since $A\cap B=\emptyset$, every $(A,B)$-path in $J_i$ contains a subpath from $N_{\widetilde{G}_{i-1}}(A)$ to $B$. Let $P'=(b_2,\dots,b_q)$ be a geodesic $(N_{\widetilde{G}_{i-1}}(A),B)$-path in $J_i$. Then $N_{\widetilde{G}_{i-1}}[P']\cap A\subseteq N_{\widetilde{G}_{i-1}}[\{b_2\}]$. Choose $b_1\in A$ adjacent to $b_2$, and let $P:=b_1P'$ be the resulting $(A,B)$-path in $\widetilde{G}_{i-1}$.

By the construction of $P$, 
$$N_{\widetilde{G}_{i-1}}[P]\subseteq N_{\widetilde{G}_{i-1}}[\{b_1,b_2\}]\cup (N_{\widetilde{G}_{i-1}}[J_i]\setminus A).$$

 By the definition of $J_i$, $\widetilde{Q}_i$ and $\mathcal{X}_{i-1}$, we have 
 $$N_{\widetilde{G}_{i-1}}[J_i]\setminus A=N_{\widetilde{G}_{i-1}}[V(\mathcal{W}^{\ast})\cup V(Q_i)]\setminus A\subseteq N_{\widetilde{G}_{i-1}}[\mathcal{X}_{i-1}]\cup \widetilde{Q}_i.$$
Consequently,
$$N_{\widetilde{G}_{i-1}}[P]\subseteq N_{\widetilde{G}_{i-1}}[\{b_1,b_2\}]\cup N_{\widetilde{G}_{i-1}}[\mathcal{X}_{i-1}]\cup \widetilde{Q}_i.$$

Now $\alpha(N_{\widetilde{G}_{i-1}}[\{b_1,b_2\}])\leq 2t$ since $\widetilde{G}_{i-1}$ is $K_{1,t}$-free. Therefore, by \cref{BasicLemSep2} and \cref{BasicLemSep3},
\begin{align*}
\asep(G-N_G[P],A,B)
&\ge\asep\left(\widetilde{G}_{i-1}-N_{\widetilde{G}_{i-1}}[P],A,B\right) \\
&\ge\asep\left(\widetilde{G}_{i-1}-\left(N_{\widetilde{G}_{i-1}}[\{b_1,b_2\}]\cup N_{\widetilde{G}_{i-1}}[\mathcal{X}_{i-1}]\cup \widetilde{Q}_i\right),A,B\right) \\
&\ge\asep\left(\widetilde{G}_{i-1}-\left(N_{\widetilde{G}_{i-1}}[\mathcal{X}_{i-1}]
\cup\widetilde{Q}_i\right),A,B\right)-2t \\
&\geq s.
\end{align*}

Thus $P$ satisfies outcome \cref{outcome_path}, contradicting our assumption
that no such path exists. Therefore
\[\tag{1}\asep\left(\widetilde{G}_{i-1}-\left(N_{\widetilde{G}_{i-1}}[\mathcal{X}_{i-1}] \cup\widetilde{Q}_i\right),A,B\right)<s+2t.\]

For each $r\in[p]$, let $Q_i^r$ denote the $(u_1,u_r)$-subpath of $Q_i$. Define
$$\widetilde{Q}_i^r:=N_{\widetilde{G}_{i-1}}[Q_i^r]\setminus\left(\bigcup(N_{\widetilde{G}_{i-1}}[W_y^\ast]\colon {y\in L^\ast})\cup\bigcup(N_{\widetilde{G}_{i-1}}[R_j]\colon {j<i})\cup A\right).$$
Then $\widetilde{Q}_i^p=\widetilde{Q}_i$. Let $r\in [p]$ be minimum such that one of the following occurs:
\begin{enumerate}[label={$(R{\arabic*}')$}, start=10]
    \item\label{R10a}
    for some $y\in L^\ast\setminus\{v\}$,
    $$ \asep\left(\left( \widetilde{G}_{i-1}- \bigl( N_{\widetilde{G}_{i-1}}[V(\mathcal{W}^\ast)\setminus W_y^\ast] \cup \bigcup(\widetilde{R}_j\colon {j<i}) \cup A \bigr)\right)-\widetilde{Q}_i^r, N_{\widetilde{G}_{i-1}}[W_y^\ast], B \right) <s_i+\Delta; $$

    \item\label{R11a}
    $$ \hspace{-6em} \asep\left(\Bigl( \widetilde{G}_{i-1}- \bigl( N_{\widetilde{G}_{i-1}}[V(\mathcal{W}^\ast)\setminus W_v^\ast] \cup N_{\widetilde{G}_{i-1}}\Bigl[ \bigl(\bigcup(R_j\colon {j<i})\bigr) \Bigr] \cup A \bigr)\Bigr)-N_{\widetilde{G}_{i-1}}[Q_i^r], N_{\widetilde{G}_{i-1}}[W_v^\ast], B \right) <s_i+\Delta; $$

    \item\label{R12a}
    for some $j<i$,
    $$ \asep\left( \left(\widetilde{G}_{i-1}- \Bigl( N_{\widetilde{G}_{i-1}}[V(\mathcal{W}^\ast)] \cup \bigcup(\widetilde{R}_b\colon {b\in[i-1]\setminus\{j\}}) \cup A \Bigr)\right)-\widetilde{Q}_i^r, \widetilde{R}_j, B \right) <s_i+\Delta; $$

    \item\label{R13a}
    $$ \asep\left( \Bigl(\widetilde{G}_{i-1}- N_{\widetilde{G}_{i-1}}[\mathcal{X}_{i-1}]\Bigr)- \widetilde{Q}_i^r, A,B \right) <s_i+\Delta. $$
\end{enumerate}

Such an $r$ exists. Indeed, for $r=p$, outcome \cref{R13a}
follows from (1), since $\widetilde{Q}_i^r=\widetilde{Q}_i$, and $s_i+\Delta\geq s+2t$ for all $i\in [N]$.

We next show that $r\geq 4$. Suppose $r\leq 3$. Then $\widetilde{Q}_i^r\subseteq N_{\widetilde{G}_{i-1}}[Q_i^r]\subseteq N_{\widetilde{G}_{i-1}}[\{u_1,u_2,u_3\}]$. Since $\widetilde{G}_{i-1}$ is $K_{1,t}$-free, $\alpha( N_{\widetilde{G}_{i-1}}[\{u_1,u_2,u_3\}])\leq 3t.$ By \cref{BasicLemSep2}, deleting either $\widetilde{Q}_i^r$ or $N_{\widetilde{G}_{i-1}}[Q_i^r]$ can reduce any of the relevant separator values by at most $3t$. Before this deletion, each of the corresponding separator values were at least $s_{i-1}$. This follows from \cref{R10}--\cref{R13} applied at $i-1$. However, since 
$$s_{i-1}-3t=(2s_i+2\Delta+3t)-3t > s_i+\Delta $$ 
 none of \cref{R10a}--\cref{R13a} can hold for $r\leq 3$. Hence $r\geq 4$. 

Set $a:=r-1.$ Then $a\geq 3$, and none of \cref{R10a}--\cref{R13a} holds with
$r=a$. Define $R_i:=Q_i^a.$ Then $R_i$ is a non-empty path in $G^{\ast}$ that is adjacent to $W_v^\ast$, hence \cref{R1} holds. Since $u_p\not\in V(R_i)$,
it follows that $R_i\cap B=\emptyset$. Furthermore, since $R_i\subseteq Q_i$, we have
$$V(R_i)\cap\left(N_{\widetilde{G}_{i-1}}[V(\mathcal{W}^\ast)\setminus W_v^\ast]\cup W_v^{\ast}\cup A\cup N_{\widetilde{G}_{i-1}}\left[\bigcup(R_j\colon {j<i})\right]
\cup N_G\left[\bigcup(D_j\colon {j<i})\right]\right)=\emptyset.$$

Thus \cref{R2} and \cref{R7} hold.

Set $\mathcal{X}_i:=V(\mathcal{W}^\ast)\cup \bigcup(V(R_j)\colon {j\in[i]})$ to give \cref{R3}.

Let $\mathcal{S}_i=\{W_y^\ast:y\in L^\ast\setminus\{v\}\}\cup \{R_j:j<i\}$. Let
$D_i\subseteq N_{\widetilde{G}_{i-1}}(R_i)\setminus \mathcal{X}_i$ be minimal with the property that whenever $\bigl(N_{\widetilde{G}_{i-1}}(R_i)\cap N_{\widetilde{G}_{i-1}}(X)\bigr)\setminus \mathcal{X}_i \neq\emptyset$ for some $X\in \mathcal{S}_i$ then $N_{\widetilde{G}_{i-1}}(X)\cap D_i\neq\emptyset.$
Then \cref{R4} is satisfied. By minimality, at most one vertex is needed for each member of $\mathcal{S}_i$. Therefore
$$|D_i|\leq |\mathcal{S}_i|\leq |L^\ast|+(i-1)\leq n\ell+N.$$

Define $\widetilde{D}_i:=N_{\widetilde{G}_{i-1}}[D_i]\setminus\mathcal{X}_i$ and $\widetilde{G}_i:=G^\ast-\bigcup(\widetilde{D}_j\colon {j\in[i]}).$ This gives \cref{R5} and \cref{R6}. Since $G$ is $K_{1,t}$-free,
$$\alpha(\widetilde{D}_i)\leq |D_i|t\leq t(n\ell+N)=\Delta.$$

Suppose that some $d\in D_i$ is adjacent to some vertex of $D_j$
with $j<i$. Since $D_i\subseteq V(\widetilde{G}_{i-1})\setminus\mathcal{X}_i \subseteq V(\widetilde{G}_{i-1})\setminus\mathcal{X}_j, $ we would have $d\in N_{G^\ast}[D_j]\setminus\mathcal{X}_j \subseteq \widetilde{D}_j,$ contradicting $d\in V(\widetilde{G}_{i-1})$. Therefore $D_i$ is anti-complete to $D_j$ for every $j<i$, proving \cref{R8}.

Define
    $$ \widetilde{R}_i:= N_{\widetilde{G}_i}[R_i]\setminus \left( \bigcup(N_{\widetilde{G}_i}[W_y^\ast]\colon {y\in L^\ast}) \cup \bigcup(N_{\widetilde{G}_i}[R_j] \colon {j<i})\cup A \right). $$
 This is \cref{R9}. Note that $\widetilde{R}_i$ is defined in $\widetilde{G}_i$, whereas $\widetilde{Q}_i^a$ is defined in the graph $\widetilde{G}_{i-1}$. Since $\widetilde{G}_{i}=\widetilde{G}_{i-1}-\widetilde{D}_i$, it follows that 
 $$\widetilde{R}_i\subseteq \widetilde{Q}_i^a\subseteq \widetilde{R}_i\cup \widetilde{D}_i.$$
 It remains to prove the separator lower bounds.
Since none of \cref{R10a}--\cref{R13a} holds with $r=a$, each of those separator values in $\widetilde{G}_{i-1}$ after deleting $Q_i^a$ are at least $s_i+\Delta$. Since $\widetilde{R}_i\subseteq \widetilde{Q}_i^a$, $N_{\widetilde{G}_i}[R_i]\subseteq N_{\widetilde{G}_{i-1}}[Q_i^a]$, $\widetilde{G}_i:=\widetilde{G}_{i-1}-\widetilde{D}_i$, and $\alpha(\widetilde{D}_i)\leq \Delta$, \cref{BasicLemSep2} 
gives the lower bounds for \cref{R10}, \cref{R11}, \cref{R12} for $j<i$, and \cref{R13}.

It remains to prove \cref{R12} for $j=i$. That is, we need to show
 $$ \asep\left( \widetilde{G}_i- \left( N_{\widetilde{G}_i}[V(\mathcal{W}^\ast)] \cup \bigcup(\widetilde{R}_b\colon {b\in[i-1]})\cup A \right), \widetilde{R}_i, B \right) \geq s_i.$$
 
Recall that $a\in [p]$ was chosen so that none of \cref{R10a}--\cref{R13a} holds for $r=a$, while at least one of them holds for $r=a+1$. We first prove the following auxiliary bound according to which of \cref{R10a}--\cref{R13a} holds for $r=a+1$.
\[\tag{(*)} \asep\left( \widetilde{G}_{i-1}- \left( N_{\widetilde{G}_{i-1}}[V(\mathcal{W}^\ast)] \cup \bigcup(\widetilde{R}_j\colon {j\in[i-1]}) \cup A \right), \widetilde{Q}_i^{a+1}, B \right) \geq s_i+\Delta+t.\]
        
\textbf{Case 1: \cref{R10a} holds.}\\
        Let $y\in L^{\ast}\setminus\{v\}$ be the witness for \cref{R10a}. By \cref{R10} at stage $i-1$ applied to $y$ together with \cref{R10a}, we have
        \begin{align*}
              \asep\Bigl(& \widetilde{G}_{i-1} - \Bigl(N_{\widetilde{G}_{i-1}}\bigl[V(\mathcal{W}^{\ast})\setminus W_y^{\ast}\bigr] \cup \bigcup(\widetilde{R}_j\colon {j\in[i-1]}) \cup A \Bigr), N_{\widetilde{G}_{i-1}}\bigl[W_y^{\ast}\bigr], B \Bigr) \\
              & - \asep\Bigl(\Bigl( \widetilde{G}_{i-1}- \bigl( N_{\widetilde{G}_{i-1}}\bigl[V(\mathcal{W}^{\ast})\setminus W_y^{\ast}\bigr] \cup \bigcup(\widetilde{R}_j\colon {j\in[i-1]}) \cup A
                \bigr)\Bigr) - \widetilde{Q}_i^{a+1}, N_{\widetilde{G}_{i-1}}\bigl[W_y^{\ast}\bigr], B \Bigr) \\
              &\geq s_{i-1}-(s_i+\Delta) = (2s_i+2\Delta+3t)-(s_i+\Delta)= s_i+\Delta+3t.
    \end{align*}

        Thus the drop in the separator value caused by deleting $\widetilde{Q}_i^{a+1}$ is at least $s_i+\Delta+3t$.
        Since $\widetilde{Q}_i^{a+1}\cap N_{\widetilde{G}_{i-1}}[W_y^\ast]=\emptyset$, \cref{BasicLemSep1} gives ((*)).

    \textbf{Case 2: \cref{R11a} holds.}\\ 
        By the definition of $H_i$, \cref{R11a} is equivalent to
        $$\asep(H_i-N_{H_i}[Q_i^{a+1}],N_{H_i}[W_v^\ast],B)<s_i+\Delta.$$
        Since $Q_i=(u_1,u_2,\dots,u_p)$ is a geodesic $(N_{H_i}(W_v^\ast),B)$-path in $H_i$, we have 
        $$N_{H_i}[Q_i^{a+1}]\cap N_{H_i}[W_v^\ast]\subseteq N_{H_i}[\{u_1,u_2\}].$$
        By the definition of $\widetilde{Q}_i^{a+1}$ and $H_i$, we have 
        $$N_{H_i}[Q_i^{a+1}]\subseteq  \widetilde{Q}_i^{a+1}\cup (N_{H_i}[Q_i^{a+1}]\cap  N_{H_i}[W_v^\ast])\subseteq \widetilde{Q}_i^{a+1}\cup N_{H_i}[\{u_1,u_2\}].$$ 
        Since $H_i$ is $K_{1,t}$-free, $\alpha(N_{H_i}[\{u_1,u_2\}])\leq 2t$. By applying \cref{BasicLemSep2} and \cref{BasicLemSep3}, we have
        \begin{align*}
            &\asep\left( H_i-(\widetilde{Q}_i^{a+1}\setminus N_{H_i}[\{u_1,u_2\}]), N_{H_i}[W_v^\ast], B \right)\\
            &\quad \leq  \asep\left(H_i-(N_{H_i}[Q_i^{a+1}]\setminus N_{H_i}[\{u_1,u_2\}]), N_{H_i}[W_v^\ast], B \right) \\
            &\quad \leq \asep\left(H_i-N_{H_i}[Q_i^{a+1}], N_{H_i}[W_v^\ast], B \right)+2t\\
            &\quad<s_i+\Delta+2t.
        \end{align*}
        Combining this with \cref{R11} at stage $i-1$, we have
         \begin{align*}
            &\asep\Bigl(H_i, N_{H_i}\bigl[W_v^{\ast}\bigr], B \Bigr) \\
            &\quad - \asep\left( H_i-(\widetilde{Q}_i^{a+1}\setminus N_{H_i}[\{u_1,u_2\}]), N_{H_i}[W_v^\ast], B \right) \\
            &\quad > s_{i-1}-(s_i+\Delta+2t)= s_i+\Delta+t.
        \end{align*}

         Since $(\widetilde{Q}_i^{a+1}\setminus N_{H_i}[\{u_1,u_2\}])\cap N_{H_i}[W_v^\ast]=\emptyset$, applying \cref{BasicLemSep1} and \cref{BasicLemSep3} give
          \begin{align*}
            &\asep\Bigl(H_i-N_{H_i}\bigl[W_v^{\ast}\bigr], \widetilde{Q}_i^{a+1}, B \Bigr) \\
            &\quad \geq \asep\Bigl(H_i-N_{H_i}\bigl[W_v^{\ast}\bigr], \widetilde{Q}_i^{a+1} \setminus N_{H_i}[\{u_1,u_2\}], B \Bigr) \\
            &\quad \geq s_i+\Delta+t.
        \end{align*}     
         By expanding out the definition of $H_i$, we obtain ((*)).
         
         \textbf{Case 3: \cref{R12a} holds.}\\
        Let $j\in [i-1]$ witness \cref{R12a}.
        By \cref{R12} applied at stage $i-1$ to $j$ together with \cref{R12a}, we have
         \begin{align*}
            \asep\Bigl(& \widetilde{G}_{i-1} - \bigl( N_{\widetilde{G}_{i-1}}\bigl[V(\mathcal{W}^{\ast})\bigr] \cup \bigcup(\widetilde{R}_b\colon {b\in[i-1]\setminus\{j\}}) \cup A \bigr), \widetilde{R}_j, B \Bigr) \\
            &- \asep\Bigl(\Bigl( \widetilde{G}_{i-1} - \bigl( N_{\widetilde{G}_{i-1}}\bigl[V(\mathcal{W}^{\ast})\bigr] \cup \bigcup(\widetilde{R}_b\colon {b\in[i-1]\setminus\{j\}})\cup A \bigr) \Bigr)-\widetilde{Q}_i^{a+1}, \widetilde{R}_j, B \Bigr) \\
            &> s_{i-1}-(s_i+\Delta) = s_i+\Delta+3t.
            \end{align*}
       Thus the drop in the separator value by deleting $\widetilde{Q}_i^{a+1}$ is at least $s_i+\Delta+3t$. Since $\widetilde{Q}_i^{a+1}\cap \widetilde{R}_j=\emptyset$, \cref{BasicLemSep1} gives ((*)).
        
        \textbf{Case 4: \cref{R13a} holds.}\\
         Combining \cref{R13} at stage $i-1$ with \cref{R13a}, we have
        \begin{align*}
            \asep&\left( \widetilde{G}_{i-1}- N_{\widetilde{G}_{i-1}}[\mathcal{X}_{i-1}], A,B \right)\\
            &-\asep\left( \Bigl(\widetilde{G}_{i-1}- N_{\widetilde{G}_{i-1}}[\mathcal{X}_{i-1}]\Bigr)
             - \widetilde{Q}_i^{a+1}, A,B \right) \\
            & \geq s_{i-1}-(s_i+\Delta) =s_i+\Delta+3t.
        \end{align*}
         So again, the drop in the separator value by deleting $\widetilde{Q}_i^{a+1}$ is at least $s_i+\Delta+3t$.  Since $\widetilde{Q}_i^{a+1}\cap A=\emptyset$, \cref{BasicLemSep1} gives us ((*)).

        Thus, we have verified ((*)) in each of the four cases. We now deduce \cref{R12} for $j=i$. Recall that $\widetilde{G}_i:=\widetilde{G}_{i-1}-\widetilde{D}_i$ and $\alpha(\widetilde{D}_i)\leq \Delta$. Moreover, observe that $\widetilde{Q}_i^{a+1}\setminus\widetilde{D}_i\subseteq \widetilde{R}_i\cup N_{\widetilde{G}_{i-1}}[u_{a+1}]$. Since $G$ is $K_{1,t}$-free, $\alpha(N_{\widetilde{G}_{i-1}}[u_{a+1}])\leq t$. By applying \cref{BasicLemSep1,BasicLemSep2}, we have
        \begin{align*}
            \asep&\left( \widetilde{G}_i- \left( N_{\widetilde{G}_i}[V(\mathcal{W}^\ast)] \cup \bigcup(\widetilde{R}_b\colon {b<i}) \cup A \right), \widetilde{R}_i, B \right) \\
            &=  \asep\left( (\widetilde{G}_{i-1}-\widetilde{D}_i)- \left( N_{\widetilde{G}_{i-1}}[V(\mathcal{W}^\ast)] \cup \bigcup(\widetilde{R}_b\colon {b<i}) \cup A \right), \widetilde{R}_i, B \right)  \\
            &\geq  \asep\left((\widetilde{G}_{i-1}-\widetilde{D}_i)- \left( N_{\widetilde{G}_{i-1}}[V(\mathcal{W}^\ast)] \cup \bigcup(\widetilde{R}_b\colon {b<i}) \cup A \right),  (\widetilde{R}_i\cup N_{\widetilde{G}_{i-1}}[u_{a+1}]), B \right)  -t \\
            &\geq  \asep\left((\widetilde{G}_{i-1}-\widetilde{D}_i)- \left( N_{\widetilde{G}_{i-1}}[V(\mathcal{W}^\ast)] \cup \bigcup(\widetilde{R}_b\colon {b<i}) \cup A \right), \widetilde{Q}_i^{a+1}, B \right)  -t \\
            &\geq  \asep\left(\widetilde{G}_{i-1}- \left( N_{\widetilde{G}_{i-1}}[V(\mathcal{W}^\ast)] \cup \bigcup(\widetilde{R}_b\colon {b<i}) \cup A \right), \widetilde{Q}_i^{a+1}, B \right)  -(\Delta+t) \\
            &\geq  (s_i+\Delta+t)-(\Delta+t)  =s_i .
        \end{align*}
        This proves \cref{R12} for $j=i$. Thus all properties \cref{R1}--\cref{R13} hold at stage $i$. This completes the induction on $i$, and hence proves the claim.
\end{proof}

    Set $\widetilde{G}:=\widetilde{G}_N$. At this stage, the branch sets of $\mathcal{W}^{\ast}$ together with the paths $R_1,\dots, R_N$ form a preliminary induced minor model of a tree obtained from $T$ by cloning each leaf. Properties \cref{W1}, \cref{W3}, and \cref{W4}, are already built into the construction. It remains to choose the branch sets for the leaves of $T$ so that \cref{W2} holds. 
    
    For the next two claims, we may assume that $G$ is $K_{c,c}$-induced-minor-free, otherwise outcome \cref{outcome_Knn} holds.
    
    \begin{claim}\label{Claimleavesbetween}
        There exists a set $S_1\subseteq [N]$ with $|S_1|=\widetilde{N}$ such that $(N_{\widetilde{G}}(R_i)\cap N_{\widetilde{G}}(R_j))\setminus\mathcal{X}_N=\emptyset$ for all distinct $i,j\in S_1$.
    \end{claim}

    \begin{proof} 
         Define an auxiliary graph $J$ with vertex set $[N]$, where distinct vertices $i$ and $j$ are adjacent if and only if $(N_{\widetilde{G}}(R_i)\cap N_{\widetilde{G}}(R_j))\setminus \mathcal{X}_N\neq \emptyset$. By the construction of $J$, it is enough to show that it contains an independent set of size $\widetilde{N}$.

         We first prove that $J$ contains no clique of size $2n$ (recall that $n:=f_C(c,c,t)$). Suppose that such a clique exists. Let $i_1<i_2<\dots <i_{2n}$ be the indices of the clique. We claim that $(R_{i_1},\dots,R_{i_n},R_{i_{n+1}},\dots, R_{i_{2n}},D_{i_1},\dots,D_{i_n})$ is an $n$-array in $G$. Indeed:
         \begin{itemize}
            \item By \cref{R1} and \cref{R7}, the subgraphs $R_{i_1},\dots,R_{i_{2n}}$ are connected and pairwise anti-complete in $G$, giving \cref{C1}.
            \item By \cref{R4}, we have $D_{i_{a}}\subseteq N_G[R_{i_{a}}]$ for every $a\in [n]$, giving \cref{C2}.
            \item Since $\{{i_{1}},\dots,{i_{2n}}\}$ is a clique in $J$, for every $a,b\in [n]$, we have $\bigl(N_{\widetilde{G}}(R_{i_{a}})\cap N_{\widetilde{G}}(R_{{i_{n+b}}})\bigr)\setminus \mathcal{X}_N\neq \emptyset$. By \cref{R4}, $D_{i_{a}}$ has a neighbour in $R_{i_{n+b}}$, giving \cref{C3}.
            \item By \cref{R7} and \cref{R8}, for all $1\leq a<b\leq n$, the set $D_{i_{a}}$ is anti-complete to $R_{i_{b}}\cup D_{i_{b}}$, giving \cref{C4}. 
        \end{itemize}
         
        By \cref{array}, $G$ contains $K_{c,c}$ as an induced minor, a contradiction. Thus $J$ has no clique of size $2n$. Since $N=R(\widetilde{N},2n)$, Ramsey's Theorem implies that $J$ contains an independent set $S_1$ of size $\widetilde{N}$, as required.
    \end{proof}

    \begin{claim}\label{Claimz}
        There exists $S_2\subseteq S_1$ with $|S_2|=|L_v|$ such that, for every $z\in L\setminus L_v$, there is a clone $\widehat{z}\in L^{\ast}$ of $z$ satisfying $(N_{\widetilde{G}}(W^{\ast}_{\widehat{z}})\cap N_{\widetilde{G}}(R_i))\setminus \mathcal{X}_N=\emptyset$ for every $i\in S_2$.
    \end{claim}
       \begin{proof} 
          Let $z_1,\dots, z_q$ be an ordering of $L\setminus L_v$. Then $q\leq |L|= \ell$. For each $a\in [q]$, let $\mathcal{C}_a:=\{z_a^{(1)},\dots,z_a^{(n)}\}\subseteq L^{\ast}$ be the set of the $n$ clones of $z_a$ in $T^\ast$. Set $\mathcal{C}:=\bigcup(\mathcal{C}_a\colon a\in [q])$. Define a bipartite graph $H$ with bipartition $(S_1,\mathcal{C})$ where $i\in S_1$ is adjacent to $z_a^{(j)}\in \mathcal{C}$ if and only if 
          $(N_{\widetilde{G}}(R_i)\cap N_{\widetilde{G}}(W_{z_a^{(j)}}^\ast))\setminus \mathcal{X}_N\neq \emptyset.$
          
          For the sake of contradiction, suppose that $H$ contains $K_{n,n}$ as a subgraph. Let $i_1<i_2<\dots<i_n$ be the indices on the $S_1$-side of $K_{n,n}$, and let $z_{a_1}^{(j_1)},\dots,z_{a_n}^{(j_n)}$ be the clones on the $\mathcal{C}$-side. Then
          $$(R_{i_1},\dots,R_{i_n}, W_{z_{a_1}^{(j_1)}}^\ast,\dots,W_{z_{a_n}^{(j_n)}}^\ast,  D_{i_1},\dots,D_{i_n})$$
          is an $n$-array in $G$. Indeed:
          \begin{itemize}
            \item Since $W_{z_{a_1}^{(j_1)}}^{\ast},\dots,W_{z_{a_n}^{(j_n)}}^{\ast}$ are branch sets corresponding to leaves in an induced minor model of a tree, they are non-empty connected subgraphs that are pairwise anti-complete. By \cref{R1} and \cref{R7}, the subgraphs $R_{i_1},\dots,R_{i_n}$ are connected and pairwise anti-complete in $G$. By \cref{R2}, $R_{i_b}$ and $W_{z_{a_{b'}}^{(j_{b'})}}^{\ast}$ are pairwise anti-complete for all $b,b'\in [n]$, giving \cref{C1}.
            \item By \cref{R4}, we have $D_{i_a}\subseteq N_G[R_{i_a}]$ for every $a\in [n]$, giving \cref{C2}.
            \item By the construction of $H$, for every $b,d\in [n]$, we have $\bigl(N_{\widetilde{G}}(R_{i_b})\cap N_{\widetilde{G}}(W_{z_{a_d}^{(j_d)}}^{\ast})\bigr)\setminus \mathcal{X}_N\neq \emptyset$. Hence, by \cref{R4}, $D_{i_b}$ has a neighbour in $V(W_{z_{a_d}^{(j_d)}}^{\ast})$, giving \cref{C3}.
            \item By \cref{R7} and \cref{R8}, for all $1\leq b<b'\leq n$, the set $D_{i_b}$ is anti-complete to $R_{i_{b'}}\cup D_{i_{b'}}$, giving \cref{C4}. 
        \end{itemize}
         By \cref{array}, this yields $K_{c,c}$ as an induced minor of $G$, a contradiction. Hence $H$ is $K_{n,n}$-free. Since $|S_1|=\widetilde{N}=f_R(\ell,n)$ and $|\mathcal{C}_a|=n$ for every $a\in [q]$, it follows by \cref{BipartiteRamsey} that there exists $S_2'\subseteq S_1$ with $|S_2'|\geq\ell$ such that, for every $a\in [q]$, there is a clone $\widehat{z}_a\in \mathcal{C}_a$ with no neighbour in $S_2'$. Setting $S_2$ to be a subset of $S_2'$ with $|S_2|=|L_v|$, the claim follows
         by the construction of $H$.
\end{proof}    
    
    We are now ready to define our induced minor model of $T$. For every non-leaf $y\in M$, set $W_y:=W_y^{\ast}$. For every leaf $z\in L\setminus L_v$, set $W_z:=W_{\widehat{z}}^{\ast}$ where $\widehat{z}$ is the clone of $z$ given by \cref{Claimz}. Finally, choose an arbitrary bijection $\phi\colon L_v\to S_2$ and set $W_z:=R_{\phi(z)}$ for each $z\in L_v$. The next claim completes the proof of the induction hypothesis.

    \begin{claim}\label{ClaimModel}
        $\mathcal{W}=(W_y\colon y\in V(T))$ is an induced minor model of $(T,x)$ in $\widetilde{G}$ satisfying \cref{W1}--\cref{W4}.
    \end{claim}
    
   \begin{proof}
         We shall repeatedly use the fact that
         $$V(\mathcal{W})\subseteq V(\mathcal{W}^{\ast})\cup \bigcup(V(R_j)\colon j\in S_2)\subseteq\mathcal{X}_N.$$
        We first verify that $\mathcal{W}$ is an induced minor model of $T$ in $\widetilde{G}$. By \cref{R5} and \cref{R6}, it follows that $\mathcal{X}_N\subseteq V(\widetilde{G})$ and thus $V(\mathcal{W})\subseteq V(\widetilde{G})$.
        Since $\mathcal{W}^{\ast}$ is an induced minor model of $T^{\ast}$, and $W_y\in \mathcal{W}^{\ast}$ for all $y\in V(T)\setminus L_v$, it follows that $(W_y\colon y\in V(T)\setminus L_v)$ is an induced minor model of $T-L_v$. For $z\in L_v$, \cref{R1} implies that $W_z$ is non-empty, connected and adjacent to $W_v$. \cref{R2} implies that $W_z$ is anti-complete to $W_y$ for all $y\in V(T)\setminus (L_v\cup \{v\})$ while also being vertex disjoint from $W_v$. Finally, \cref{R7} implies that $W_z$ is anti-complete to $W_{y}$ for all $y\in L_v\setminus\{z\}$. Thus $\mathcal{W}$ is an induced minor model of $T$ in $\widetilde{G}$.

        We now verify that the model satisfies \cref{W1}--\cref{W4}. By induction, $W_x^{\ast}\cap A\neq \emptyset$ and $V(\mathcal{W}^{\ast})\cap B=\emptyset$. Since $W_x=W_x^{\ast}$, \cref{W3} is satisfied. Since $\bigcup(V(R_j)\colon j\in [N])\cap B=\emptyset$ by \cref{R2}, it follows that $V(\mathcal{W})\cap B=\emptyset$, thus \cref{W4} is satisfied. 
        
        Applying \cref{R13} with \cref{BasicLemSep3}, we have
        \begin{align*}
            \asep\bigl(\widetilde{G}-N_{\widetilde{G}}[V(\mathcal{W})],A,B\bigr)
            &\geq \asep\bigl(
            \widetilde{G}_N-N_{\widetilde{G}_N}[\mathcal{X}_N],A,B\bigr) \\
            &\geq s_N \geq s.
        \end{align*}
        Thus \cref{W1} is satisfied.
        
        It remains to show \ref{W2}. By \cref{R9}, we have 
         $$N_{\widetilde{G}}[\bigcup(W_y\colon y\in L_v)]=N_{\widetilde{G}}[\bigcup(R_j\colon j\in S_2)]\subseteq N_{\widetilde{G}}[V(\mathcal{W}^{\ast})]\cup \bigcup (\widetilde{R}_j\colon j\in [N]) \cup A.$$
        Let $z\in L\setminus L_v$. Let $\widehat{z}\in L^{\ast}$ be the clone of $z$ such that $W_z=W_{\widehat{z}}^{\ast}$. By \cref{Claimz}, we have
        $$\left(N_{\widetilde{G}}[\bigcup(R_j\colon j\in S_2)]\cap N_{\widetilde{G}}[W_{\widehat{z}}^{\ast}]\right)\setminus \mathcal{X}_N=\emptyset.$$ 
        By \cref{R2}, $W_{\widehat{z}}^{\ast}$ is anti-complete to $R_j$ for all $j\in S_2$. Hence
         $$N_{\widetilde{G}}[\bigcup(W_y\colon y\in L_v)]\subseteq N_{\widetilde{G}}[V(\mathcal{W}^{\ast})\setminus W_{\widehat{z}}^{\ast}]\cup  \bigcup(\widetilde{R}_j\colon j\in [N])\cup A.$$
         As such,
        $$N_{\widetilde{G}}[V(\mathcal{W})\setminus W_z]\subseteq N_{\widetilde{G}}[V(\mathcal{W}^{\ast})\setminus W_{\widehat{z}}^{\ast}]\cup N_{\widetilde{G}}[\bigcup(W_y\colon y\in L_v)] \subseteq N_{\widetilde{G}}[V(\mathcal{W}^{\ast})\setminus W_{\widehat{z}}^{\ast}]\cup \bigcup(\widetilde{R}_j\colon j\in [N])\cup A.$$
        By \cref{R10} and \cref{BasicLemSep3}, we have,
        \begin{align*}
            \asep(\widetilde{G}-(N_{\widetilde{G}}[V(\mathcal{W})\setminus W_z]\cup A),&N_{\widetilde{G}}[W_z],B)\\
            &\geq  \asep(\widetilde{G}-(N_{\widetilde{G}}[V(\mathcal{W}^{\ast})\setminus W_{\widehat{z}}^{\ast}]\cup \bigcup(\widetilde{R}_j\colon j\in [N])\cup A),N_{\widetilde{G}}[W_{\widehat{z}}^{\ast}],B)   \\       
            &\geq s_N \geq s .
        \end{align*}
        Thus \cref{W2} holds for every $z\in L\setminus L_v$.
        
        Now consider $y\in L_v$. By \cref{R7}, $R_{\phi(y)}$ is anti-complete to $R_a$ for all $a\in S_2\setminus \{{\phi(y)}\}$. Thus, by \cref{Claimleavesbetween},
        $$\left(N_{\widetilde{G}}[\bigcup(R_a\colon a\in S_2\setminus\{{\phi(y)}\})]\cap N_{\widetilde{G}}[R_{\phi(y)}]\right)\setminus \mathcal{X}_N=\emptyset.$$
        As such, we have
        $$N_{\widetilde{G}}[\bigcup(W_a\colon a\in L_v\setminus\{y\})]\subseteq N_{\widetilde{G}}[V(\mathcal{W}^{\ast})]\cup  \bigcup(\widetilde{R}_a\colon a\in [N]\setminus\{{\phi(y)}\}) \cup A.$$
        Hence
         $$N_{\widetilde{G}}[V(\mathcal{W})\setminus W_y]\subseteq N_{\widetilde{G}}[V(\mathcal{W}^{\ast})]\cup \bigcup(\widetilde{R}_a\colon a\in [N]\setminus\{{\phi(y)}\})\cup A.$$
        By \cref{R9}, $\widetilde{R}_{\phi(y)}\subseteq N_{\widetilde{G}}[W_y]$. Therefore, by \cref{R12} and \cref{BasicLemSep3},
        \begin{align*}
            \asep(\widetilde{G}-(N_{\widetilde{G}}[V(\mathcal{W})\setminus W_y]\cup A),&N_{\widetilde{G}}[W_y],B)\\
            &\geq \asep(\widetilde{G}-(N_{\widetilde{G}}[V(\mathcal{W}^{\ast})]\cup \bigcup(\widetilde{R}_a\colon a\in [N]\setminus\{{\phi(y)}\})\cup A),\widetilde{R}_{\phi(y)},B)\\
            &\geq s_N\geq s.
        \end{align*}
        Thus \cref{W2} holds for every $y\in L_v$ as well. We have verified \cref{W1}--\cref{W4}, completing the proof of the claim.
   \end{proof}
   \cref{ClaimModel} gives outcome \cref{outcome_model}. This completes the induction step, and hence the proof of the lemma.
\end{proof}

\subsection{Proof of \cref{AlphaMengerForestMinor}}\label{SecProofAlphaMenger}

We say that a graph $H$ is a \defn{special apex-tree} if there is a vertex $z\in V(H)$ such that $H-z$ is a tree and $N_H(z)$ is a subset of the leaves of $H-z$. 

\begin{lem}\label{specialapex}
    For every $n\in \NN$, every $n$-vertex apex-forest $H$ is an induced minor of a special apex tree $J$ where $|V(J)|\leq 2n+1$.
\end{lem}

\begin{proof}
    Let $a\in V(H)$ be such that $F:=H-a$ is a forest. Let $T_1,\dots, T_m$ be the components of $F$. For each $i\in [m]$, choose a vertex $r_i\in V(T_i)$. Let $T$ be obtained from $F$ by adding a new vertex $x$ adjacent to each of $r_1,r_2,\dots,r_m$. Then $T$ is a tree that contains $F$ as an induced subgraph. For each $v\in N_H(a)$, add to $T$ a new leaf $v^{\ast}$ that is adjacent only to $v$. Finally, add a new vertex $z$ that is adjacent to the vertices $v^{\ast}$ for each $v\in N_H(a)$. Let $J$ be the resulting graph. Since $J-z$ is a tree and $N_J(z)$ is a subset of the leaves of $J-z$, it follows that $J$ is a special apex tree. Furthermore, we have $|V(J)|\leq |V(F)|+|N_J(z)|+|\{x,z\}|\leq 2n+1$. Finally, $J$ contains $H$ as an induced minor. First delete the vertex $x$ from $J$ then, for each $v\in N_H(a)$, contract the edge $vv^{\ast}$. The resulting graph is isomorphic to $H$, as required.
\end{proof}

\begin{lem}\label{LemIteratePaths}
    There exists a function $f_A$ such that, for all $k\in \NN_0$ and $n,s,t\in \NN$ where $s>2t$, the following holds. Let $H$ be a special apex-tree on $n+1$ vertices. Let $G$ be a $K_{1,t}$-free graph that excludes $H$ as an induced minor. Let $A,B\subseteq V(G)$, where $G[B]$ is connected. If $\asep(G,A,N_G[B])\geq f_A(k,n,s,t)$ then $G$ contains $k$ pairwise anti-complete $(A,N_G[B])$-paths $P_1,\dots, P_k$ such that 
    $$\asep(G-N_G[\bigcup(P_i \colon i\in [k])],A,N_G[B])\geq s.$$
\end{lem}

\begin{proof}
    Let $c:=f_B(n+1,t)$ where $f_B$ is the function from \cref{ExcludedBipartite}. By \cref{ExcludedBipartite}, $G$ is $K_{c,c}$-induced-minor-free. Define 
    $$\widehat{f}_{\mathcal{W}}(c,n,s,t):=\max\{f_{\mathcal{W}}(c,m',\ell',s,t)\colon m'+\ell'=n\text{ and $(m',\ell')$ is a tree-pair}\}$$ 
    where $f_{\mathcal{W}}$ is the function from \cref{TechnicalInducedMenger}. 
    Set $f_A(0,n,s,t):=s$ and, for $k\geq 1$, set
    $$f_A(k,n,s,t):=f_A(k-1,n,\widehat{f}_{\mathcal{W}}(c,n,s,t),t).$$ 
    
    We prove the lemma by induction on $k$. The case $k=0$ is immediate. So assume that $k>0$, the claim holds for $k-1$, and that $\asep(G,A,N_G[B])\geq f_A(k,n,s,t)$.

    By the induction hypothesis applied with $k-1$ and parameter $\widehat{f}_{\mathcal{W}}(c,n,s,t)$, there are $k-1$ pairwise anti-complete $(A,N_G[B])$-paths $P_1,\dots, P_{k-1}$ in $G$ such that, for $G':=G-N_G[\bigcup(P_i\colon i\in [k-1])]$, we have 
    $$\asep(G',A,N_{G}[B])\geq \widehat{f}_{\mathcal{W}}(c,n,s,t).$$
    Let $z\in V(H)$ be such that $T:=H-z$ is a tree on $n$ vertices and $N_H(z)$ is a subset of the leaves of $T$. Fix a root $x\in V(T)$. Let $M$ and $L$ be respectively the set of non-leaves and leaves in $T$. Set $m=|M|$ and $\ell=|L|$. Apply \cref{TechnicalInducedMenger} to $G'$ with terminal sets $A$ and $N_{G}[B]$ and with the rooted tree $(T,x)$. Since $G'$ is an induced subgraph of $G$, it is still $K_{1,t}$-free and $K_{c,c}$-induced-minor-free. Thus one of the following outcomes occurs:
    \begin{enumerate}[label={$(O{\arabic*}')$}]
        \item\label{O1'} there is an induced subgraph $\widetilde{G}$ of $G'$ containing an induced minor model $\mathcal{W}=(W_y\colon y\in V(T))$ of $(T,x)$ satisfying \cref{W1}--\cref{W4};
        \item\label{O2'} there is an $(A,N_{G}[B])$-path $P_k$ in $G'$ such that $\asep(G'-N_G[P_k],A,N_{G}[B])\geq  s$; or
        \item\label{O3'} $G'$ contains $K_{c, c}$ as an induced minor.
    \end{enumerate} 
    Outcome \cref{O3'} is impossible since $G'$ is $K_{c,c}$-induced-minor-free. If \cref{O2'} occurs, then $P_k\subseteq G'$, so it is anti-complete to $P_1,\dots, P_{k-1}$. Hence $P_1,\dots, P_k$ are pairwise anti-complete $(A,N_G[B])$-paths in $G$, satisfying the desired conclusion. 
    
    It remains to rule out \cref{O1'}. Suppose that \cref{O1'} occurs. We shall extend the model of $T$ into an induced minor model of $H$, contradicting the assumption that $G$ excludes $H$. 
    
    For each $y\in N_H(z)$, let $P_y$ be a geodesic $(N_{\widetilde{G}}(W_y),N_{\widetilde{G}}[B])$-path in $\widetilde{G}-N_{\widetilde{G}}[V(\mathcal{W})\setminus W_y]$. Since $s>0$ and $N_H(z)\subseteq L$, property \cref{W2} applied with $N_{G'}[B]$ in place of $B$ implies that such a path exists. Moreover, $P_y$ is anti-complete to $V(\mathcal{W})\setminus W_y$. Since $N_{\widetilde{G}}[B]\cap W_y=\emptyset$ by \cref{W4}, $P_y$ is vertex-disjoint from $W_y$. 
    
    Define $W_z:=G[B\cup \bigcup (V(P_y)\colon y\in N_H(z))]$. Since $G[B]$ is connected and each $P_y$ is adjacent to $B$, it follows that $W_z$ is a connected subgraph of $G$. By \cref{W4}, $N_{G}[B]\cap V(\mathcal{W})=\emptyset$ and so $B$ is anti-complete to $V(\mathcal{W})$. Thus $W_z$ is vertex-disjoint from $\mathcal{W}$, and adjacent only to the branch set indexed by vertices in $N_H(z)$. Therefore $\mathcal{W}\cup \{W_z\}$ is an induced minor model of $H$ in $G$, contradicting the assumption that $G$ is $H$-induced-minor-free. Hence \cref{O1'} cannot occur, thus completing the proof. 
\end{proof}

We now deduce \cref{AlphaMengerForestMinor} from \cref{LemIteratePaths}. The next theorem is a restatement of \cref{AlphaMengerForestMinor}.

\begin{thm}
    For every apex-forest $H$ and every $t\in \NN$, there is a function $f_{H,t}$ such that the following holds. For every $K_{1,t}$-free $H$-induced-minor-free graph $G$, for every $k\in \NN$, for any pair of induced paths $P_1,P_2$ in $G$, if $\asep(G,N_G[P_1],N_G[P_2])\geq f_{H,t}(k)$, then $G$ contains $k$ pairwise anti-complete $(N_G[P_1],N_G[P_2])$-paths. 
\end{thm}

\begin{proof}
    Let $n:=|V(H)|$. By \cref{specialapex}, there exists a special apex tree $J$ with $|V(J)|\leq 2n+1$ such that $H$ is an induced minor of $J$. Thus $G$ is $J$-induced-minor-free. Define $f_{H,t}(k):=\max\{f_A(k,n',2t+1,t)\colon n'\in [2n]\}$ where $f_A$ is the function from \cref{LemIteratePaths}. By setting $A:=N_G[P_1]$ and $B:=P_2$, it follows from \cref{LemIteratePaths} that if $\asep(G,N_G[P_1],N_G[P_2])\geq f_{H,t}(k)$, then $G$ contains $k$ pairwise anti-complete $(N_G[P_1],N_G[P_2])$-paths, as required.
\end{proof}

\section{Erd\H{o}s--P\'osa and Tree Independence Number}

Graphs of bounded treewidth have a useful Erd\H{o}s--P\'osa type property: for any collection $\HH$ of connected subgraphs in a graph $G$ of bounded treewidth, $G$ contains many pairwise vertex-disjoint elements of $\HH$, or there exists a set $X\subseteq V(G)$ of bounded size such that $G-X$ contains no element of $\HH$ (see \cite[(8.7)]{robertson1986planar}). 

The next theorem extends this property to graphs of bounded tree independence number, with `vertex-disjointness' replaced by `anti-completeness'. This will be a key tool used in the next section where we convert a tree-decomposition into a path-decomposition. Recall that $R(r,s)$ is the Ramsey function (see \cref{RamseyTheorem}).

\begin{thm}\label{HittingSet}
    For all $d,k,w\in \NN$, for every graph $G$ with $\atw(G)< k$, for any collection $\HH$ of connected induced subgraphs in $G$ where $\omega(H)< w$ for each $H\in \HH$, one of the following occurs:
    \begin{enumerate}
        \item $G$ contains $d$ pairwise anti-complete elements of $\HH$; or
        \item There exists a set $X\subseteq V(G)$ with $|X|< (k+R(k,w))(d-1)$ such that $G-N_G[X]$ contains no element of $\HH$.
    \end{enumerate}
\end{thm}

\begin{proof}
    We proceed by induction on $d$. If $\HH$ is empty, then the second outcome occurs trivially for all $d$ by setting $X=\emptyset$. If $\HH$ is non-empty, then any element of $\HH$ certifies the first outcome for $d=1$. Thus we may assume that $d\geq 2$ and $\HH\neq \emptyset$. 
    
    Let $(T,B_x\colon x\in V(T))$ be a tree-decomposition of $G$ such that $\alpha(B_x)< k$ for all $x\in V(T)$. Fix a root $r\in V(T)$. For each $H\in \HH$, define $C_H:= T[\{x\in V(T)\colon V(H)\cap B_x\neq \emptyset\}]$. Since $H$ is connected, $C_H$ is a connected subtree of $T$. Choose $\widetilde{H}\in \HH$ such that $\dist_T(r,C_{\widetilde{H}})$ is maximised and let $z\in V(C_{\widetilde{H}})$ be the unique node attaining this distance.

    Let $X_1$ be a maximal independent set in $B_z$. Then $|X_1|< k$ and $B_{z}\subseteq N_G[X_1]$. Let $X_2:= B_z\cap V(\widetilde{H})$. Since $\omega(\widetilde{H})< w$, we have $|X_2|< R(k,w)$ by \cref{RamseyTheorem}.
    Define $X':=X_1\cup X_2$. Then $|X'|< k+R(k,w)$. Set $G':=G-N_G[X']$.

    Suppose that $G'$ does not contain $d-1$ pairwise anti-complete elements of $\mathcal{H}$. By induction, there exists a set $X''\subseteq V(G')$ with $|X''|< (k+R(k,w))(d-2)$ such that $G'-N_{G'}[X'']$ contains no elements of $\mathcal{H}$. By setting $X=X'\cup X''$, it follows that $G-N_G[X]$ also contains no element of $\mathcal{H}$. Moreover, $|X|\leq |X'|+ |X''|< (k+R(k,w))(d-1)$, establishing the second outcome.

    Now assume that $G'$ does contain $d-1$ pairwise anti-complete elements $H^{(1)},\dots,H^{(d-1)}$ of $\mathcal{H}$. Fix $i\in [d-1]$. We claim that $\widetilde{H}$ is anti-complete to $H^{(i)}$. Since $H^{(i)}$ is contained in $G'$, we have $V(H^{(i)})\cap B_{z}=\emptyset$. Because $\widetilde{H}$ was chosen to maximise $\dist_T(r,C_{\widetilde{H}})$ and $C_{H^{(i)}}$ is connected, no node of $C_{H^{(i)}}$ is contained in a bag indexed by a descendant of $z$. By the choice of $z$, $C_{H^{(i)}}$ and $C_{\widetilde{H}}$ are vertex-disjoint and so $H^{(i)}$ and $\widetilde{H}$ are also vertex-disjoint. Moreover, if $G$ contains an edge $uv$ with $u\in V(H^{(i)})$ and $v\in V(\widetilde{H})$, then $v\in B_{z}$. Thus $v\in X_2$, forcing $u\in N_G[X']$, contradicting $H^{(i)}\subseteq G'$. Hence $\widetilde{H}$ is anti-complete to each $H^{(i)}$, yielding the first outcome.
\end{proof}

\section{Tree Decomposition to Path Decomposition}\label{ForestPathAlpha}

In this section, we prove \cref{MainThmForest}.

\MainThmForest*

Let $T_{h,d}$ denote the $d$-ary tree of height $h$. By \cref{MainThmApexForest}, every $K_{1,t}$-free graph with no $T_{h,d}$ induced minor has bounded tree independence number. Since every forest is an induced minor of some $T_{h,d}$, \cref{MainThmForest} follows immediately from the next lemma. The proof of the lemma is inspired by the main result of \cite{DHJMMW}.

\begin{lem}
    There exist functions $f,g$ where $f(d,h,k,t)\geq g(d,h,k,t)\geq t$ such that the following holds for all $d,h,k,t\in\NN$ with $d+h\geq 3$. Let $G$ be a graph with $\atw(G)< k$ and no $T_{h,d}$ induced minor. Suppose that $G$ contains a vertex $r\in V(G)$ satisfying $\alpha(N_G(r))\leq g(d,h,k,t)$ and that, for every $v\in V(G)\setminus \{r\}$, we have $\alpha(N_G[v])\leq t$. Then $G$ admits a path-decomposition with independence number at most $f(d,h,k,t)$, where the first bag contains $N_G[r]$.
\end{lem}

\begin{proof} 
    For all $d,h,k,t\in \NN$, define
    \begin{align*}
        g(d,h,k,t)&:=(k+R(k,2(d+2)(t+1)+5))dt;\\
        f(d,1,k,t)&:=2d;\\
        f(d,h,k,t)&:=f((d+2)(t+1),h-1,k,t)+2g(d,h,k,t) \quad \quad \textrm{for $h\geq 2$.}
    \end{align*}
    We make no attempt to optimise these functions.

    We prove the claim by induction on pairs $(h,|V(G)|)$ in lexicographic order. 
    Fix $d,h,k,t$, $G$, and $r$ as in the statement. 
    We may assume that $G$ is connected, since otherwise we are done by applying induction to each of its components. The statement is trivial if $|V(G)|= 1$, so assume that $|V(G)|\geq 2$. Thus $|N_G[r]|\geq 2$.
    
    For the base case, suppose that $h=1$. For each $i\in \NN_0$, let $V_i:=\{v\in V(G):\dist_G(v,r)=i\}$. So $V_0=\{r\}$ and thus $\alpha(V_0)=1\leq d$. If $\alpha(V_i)\geq d$ for some $i\in \NN$, then by contracting the connected $G[V_0\cup\dots\cup V_{i-1}]$ into a vertex and taking an independent set of size $d$ in $V_i$, we obtain $T_{1,d}$ as an induced minor, a contradiction. Hence, $\alpha(V_i)\leq d$ for each $i\in \NN_0$. Let $B_i:=V_i\cup V_{i+1}$ for each $i\in \NN_0$. Then $(B_0,B_1,\dots)$ is a path-decomposition of $G$ with independence number at most $2d$ where the first bag contains $N_G[r]$, as desired. 
    
    Now assume that $h\geq 2$ and that the lemma holds for $h-1$. Let $M:=\{v\in V(G)\colon \dist_G(v,r)=2\}$ be the second neighbourhood of $r$ in $G$. Let $\mathcal{F}$ be the set of vertex-minimal connected induced subgraphs of $G-N_G[r]$ that contain both a vertex from $M$ and $T_{h-1,(d+2)(t+1)}$ as an induced minor. 
    
    \begin{claim}
        If $\mathcal{F}$ contains $d$ pairwise anti-complete elements $S_1,\ldots,S_d$, then $G$ contains $T_{h,d}$ as an induced minor. 
    \end{claim}

    \begin{proof}
        For each $i\in [d]$, let $\mathcal{W}_i:=(W^i_x: x\in V(T_{h-1,(d+2)(t+1)}))$ be an induced minor model of $T_{h-1,(d+2)(t+1)}$ in $S_i$.  We first construct within each $S_i$ an induced minor model of $T_{h-1,d(t+1)}$ whose root branch set contains a vertex from $M$. 
        
        Let $P_i$ be a geodesic $(\mathcal{W}_i,M\cap V(S_i))$-path in $S_i$. Note that $P_i$ may be a singleton. Since $P_i$ is geodesic, only its first and second vertex may be adjacent to $\mathcal{W}_i$. Indeed, if a later vertex was adjacent to $\mathcal{W}_i$, then we contradict $P_i$ being geodesic. Furthermore, as $\mathcal{W}_i$ is an induced minor model and $\alpha(N_G[v])\leq t$ for all $v\in V(S_i)$, $P_i$ is adjacent to at most $2t$ sub-models of $T_{h-2,(d+2)(t+1)}$ which are rooted at the children of the root node of $T_{h-1,d(t+1)}$. If $P_i$ is adjacent to none of these sub-models, then it is adjacent only to the root branch set, and as such, by adding $P_i$ to the root branch set, we obtain an induced minor model of $T_{h-1,(d+2)(t+1)}$ where the root branch set contains a vertex from $M$. If $P_i$ is adjacent to one of the sub-models of $T_{h-2,(d+2)(t+1)}$, then modify the branch set of the root node of $T_{h-1,(d+2)(t+1)}$ by taking the union of the branch set together with one of these children sub-model together with $P_i$. By first removing all other branch sets that are adjacent to $P_i$ and then pruning excess branch sets, we may obtain an induced minor model of $T_{h-1,d(t+1)}$ where the root branch set of $T_{h-1,d(t+1)}$ contains a vertex from $M$. 
        
        Repeating this for all $i\in [d]$, we obtain $d$ anti-complete induced models $(Z^i_x: x\in V(T_{h-1,d(t+1)}))$ of $T_{h-1,d(t+1)}$ whose root branch sets are adjacent (in $G$) to $N_G[r]$. For each $i\in [d]$, let $z_i$ denote the root of the $i^{th}$ copy of $T_{h-1,d(t+1)}$ and let $\widetilde{Z}^i$ be its corresponding branch set. Let $\widetilde{M}\subseteq N_G(r)$ be a minimal set of vertices such that, for each $i\in [d]$, $\widetilde{M}$ and $\widetilde{Z}^i$ are adjacent. By minimality, $|\widetilde{M}|\leq d$. Since $r$ is adjacent to all vertices in $\widetilde{M}$ and is anti-complete to $G-N_G[r]$, and $\alpha(N_G[v])\leq t$ for all $v\in V(G)\setminus \{r\}$, we have $\alpha(N_G[\widetilde{M}]\cap V(G-N_G[r]))\leq dt-1$.
        
        For each $i\in [d]$, consider the induced sub-model $(Z^i_x: x\in V(T_{h-2,d(t+1)})\setminus\{z_i\})$ of the graph obtained from $T_{h-2,d(t+1)}-z_i$. Since $T_{h-2,d(t+1)}-z_i$ contains $d(t+1)$ components, the subgraph of $G$ induced by the vertices in this model has $d(t+1)$ components. Since $\alpha(N_G[\widetilde{M}]\cap V(G-N_G[r]))\leq dt-1$, for each $i\in [d]$, at least $d$ of these components are anti-complete to $\widetilde{M}$. Therefore, $G-N_G[r]$ contains $d$ anti-complete induced minor models of $T_{h-1,d}$ where only the root branch set of each model is adjacent (in $G$) to $\widetilde{M}$. By setting $G[\{r\}\cup \widetilde{M}]$ to be the root branch set, we obtain an induced minor model of $T_{h,d}$, as claimed. 
    \end{proof}
    
    Since $G$ does not contain $T_{h,d}$ as an induced minor, it follows that $\mathcal{F}$ does not contain $d$ pairwise anti-complete elements. We now exploit the vertex minimality of the elements of $\mathcal{F}$ to bound their clique number.

    \begin{claim}
        For each $S\in \mathcal{F}$, we have $\omega(S)< 2(d+2)(t+1)+5$. 
    \end{claim}
    
    \begin{proof}
        Let $S\in \mathcal{F}$. 
        Because $S$ is vertex-minimal, there is a vertex-minimal induced minor model $(W_x: x\in V(T_{h-1,(d+2)(t+1)}))$ of $T_{h-1,(d+2)(t+1)}$ in $S$ and a vertex-minimal path $P$ from $M\cap S$ to the model such that every vertex of $S$ is either in the model or in $P$.  Fix $x\in V(T_{h-1,(d+2)(t+1)})$ and let $J\subseteq V(W_x)$ be an inclusion-wise minimal vertex set such that, for each neighbour $z$ of $x$ in $T_{h-1,(d+2)(t+1)}$, $J$ and $W_z$ are adjacent in $G$. Since $x$ has degree at most $(d+2)(t+1)+1$, we have $|J|\leq (d+2)(t+1)+1$. Since $W_x$ is vertex-minimal, removing any vertex from $W_x$ means that no component of the remainder graph contains all of $J$. This implies $\omega(S)\leq |J|\leq (d+2)(t+1)+1$. As $T_{h-1,(d+2)(t+1)}$ is triangle-free and the model is induced, any clique in $S$ is contained within the union of at most two branch sets together with $V(P)$. Since at most two vertices of $P$ are adjacent to the model, we have $\omega(S)\leq 2(d+2)(t+1)+4$.
    \end{proof}
     
    Applying \cref{HittingSet}, it follows that there exists a set $X'\subseteq V(G-N_G[r])$ with $|X'|\leq g(d,h,k,t)/t$ such that $G-N_G[r]-N_{G-N_G[r]}[X']$ contains no element of $\mathcal{F}$. Choose $X\subseteq N_{G-N_G[r]}[X']$ to be a minimal set of vertices such that $G-N_G[r]-X$ contains no element of $\mathcal{F}$. Then 
    $$\alpha(X)\leq \alpha(N_{G-N_G[r]}[X'])\leq t|X'|\leq g(d,h,k,t).$$

    Let $G_1,\dots,G_p$ be the components of $G-N_G[r]-X$ that contain a vertex from $M$. If $G_i$ contains $T_{h-1,(d+2)(t+1)}$ as an induced minor for some $i\in [p]$, then $G_i$ has a vertex-minimal connected induced subgraph that contains a vertex from $M$ and $T_{h-1,(d+2)(t+1)}$ as an induced minor, contradicting our choice of $X$. Thus, no $G_i$ contains $T_{h-1,(d+2)(t+1)}$ as an induced minor. Moreover, since each $G_i$ is an induced subgraph of $G-N_G[r]$, we have $\alpha(N_{G_i}[v]) \leq t$ for all $v\in V(G_i)$. Therefore, by rooting each $G_i$ arbitrarily, we may apply induction (on $h$) to deduce that each $G_i$ has a path-decomposition $\mathcal{B}_i$ with independence number at most $f((d+2)(t+1),h-1,k,t)$.
    
    Let $J$ be the set of vertices of all components of $G-N_G[r]-X$ that have no vertex in $M$. Then $J$ and $N_G[r]$ are anti-complete.
    
    Consider a vertex $v\in X$. By the minimality of $X$, the graph $G-N_G[r]-(X\setminus\{v\})$ contains an induced connected subgraph $Y_v$ that contains $v$ and a vertex $r_v\in M$ (and contains $T_{h-1,d(t+1)}$ as an induced minor). Let $P_v$ be a $(v,r_v)$-path in $Y_v$. If $v\neq r_v$, then $P_v-\{v\}$ contains a vertex from $M$ and so it is contained in some $G_i$, and thus $P_v-\{v\}$ is vertex-disjoint from $X$ while also being anti-complete to $J$. Moreover, the subgraph of $G$ induced by
    $N_G[r]\cup \bigcup(V(P_v)\setminus \{v\}:v\in X)$ is connected and is contained in $G-X-J$. Let $G'$ be obtained from $G$ by contracting $N_G[r] \cup \bigcup(P_v\setminus\{v\}:v\in X)$ into a vertex $r'$, and deleting any remaining vertices not in $J\cup X$. Thus $V(G')=\{r'\}\cup J\cup X$. 

    We claim that $G'$ has the properties needed for us to apply induction (on $|V(G')|$).
    \begin{claim}
        The graph $G'$ satisfies:
        \begin{enumerate}[label=(\arabic*)]
            \item $|V(G')|< |V(G)|$;
            \item $\atw(G')< k$;
            \item $G'$ has no $T_{h,d}$ induced minor;
            \item $N_{G'}(r')=X$ and hence $\alpha(N_{G'}(r'))\leq g(d,h,k,t)$; and
            \item $\alpha(N_{G'}[v]) \leq t$ for all $v\in V(G')\setminus\{r'\}$.
        \end{enumerate}
    \end{claim}
    
    \begin{proof}
         Property (1) holds since $|N_G[r]|\geq 2$. Properties (2) and (3) follow from \cref{IndependentMinor} and the fact that $G'$ is an induced minor of $G$. Property (4) is immediate from the construction and Property (5) follows because, for all $v\in V(G')\setminus \{r'\}$, we have $\alpha(N_{G'}[v])\leq \alpha(N_G[v]) \leq t$. This follows because when taking an induced minor, if a vertex is not involved in an edge-contraction, then the independence number of its closed neighbourhood does not increase.
    \end{proof}
   
    By induction, $G'$ admits a path-decomposition $\mathcal{B}'$ with independence number at most $f(d,h,k,t)$ whose first bag contains $N_{G'}[r']=\{r'\}\cup X$. Set $B_0:=N_G[r]\cup X$. We construct a path-decomposition $\mathcal{B}$ of $G$ by concatenating $B_0,\mathcal{B}_1,\dots,\mathcal{B}_p$ and $\mathcal{B}'$, adding $N_G[r]$ and $X$ to every bag that comes from $\mathcal{B}_1,\dots,\mathcal{B}_p$, then removing $\{r'\}$ from all bags in $\mathcal{B}'$. The next claim completes the proof of the lemma. 

    \begin{claim}
        $\mathcal{B}$ is a path-decomposition of $G$ with independence number at most $f(d,h,k,t)$ whose first bag contains $N_G[r]$.
    \end{claim}
    
    \begin{proof}
        Clearly $N_G[r]$ is contained in the first bag $B_0$ which has independence number at most $\alpha(N_G[r]\cup X)\leq 2g(d,h,k,t)\leq f(d,h,k,t)$. We now argue that $\mathcal{B}$ is a path-decomposition of $G$. Each vertex of $G$ is contained in consecutive bags of $\BB$, specifically $N_G[r]\cup X$ appears in every bag in $B_0,\BB_1,\dots,\BB_p$ and $X$ is in the first bag of $\BB'$. Since $G_1,\dots,G_p$ are components of $G-N_G[r]-X$, the neighbourhoods of a vertex in $G_i$ is contained in $V(G_i)\cup X\cup N_G[r]$. Note also that the neighbourhood of $N_G[r]$ is contained in $V(G_1)\cup\dots V(G_p)\cup X$. Thus $\BB$ is a path-decomposition of $G$. By induction, each bag from $\BB'$ has independence number at most $f(d,h,k,t)$, and each bag that comes from some $\BB_i$ has independence number at most $f((d+2)(t+1),h-1,k,t)+2g(d,h,k,t)=f(d,h,k,t)$, as required. 
    \end{proof}
\end{proof}
\section{Conclusion}
We conclude with some open problems that arise from our work.

First, which graphs have large path independence number? The complete binary tree $T_h$ of height $h$ has pathwidth $\lceil h/2 \rceil$ and thus has path independence number at least $h/4$ since it is bipartite. \citet{DALLARD2024indep} observed that $K_{t,t}$ has tree independence number $t$, so it also has large path independence number. Since the path independence number is closed under taking induced minor, any graph containing $T_h$ or $K_{t,t}$ as an induced minor must also have large path independence number. We conjecture that these are the only obstructions. 

\begin{conj}\label{ConjInducedMinorObs}
    There exists a function $f$ such that, for all $t,h\in \NN$, every graph $G$ with no $K_{t,t}$ or $T_h$ induced minor has path independence number at most $f(t,h)$.
\end{conj}

\citet{CHS2024pathwidth} showed that every $K_c$-free graph with no $K_{t,t}$ or $T_h$ induced minor has bounded pathwidth. So \cref{ConjInducedMinorObs} would generalise their result from pathwidth to path independence number. Note that the analogous version of \cref{ConjInducedMinorObs} for tree independence number where $T_h$ is replaced by the $(k\times k)$-grid is false. In particular, \citet{CHKTW2025bipartite} showed that the family of \emph{layered wheels} has unbounded tree independence number while excluding a grid and a complete bipartite graph as induced minors. This family of graphs is known, however, to contain all forests as induced minors, so it should not be considered as evidence against our conjecture.

Our final conjecture concerns the global structure of graphs that exclude a tree as an induced minor without any additional structural assumptions being made. We define the \defn{path domination number} of a graph $G$ to be the minimum integer $k$ such that $G$ admits a path-decomposition where every bag is contained within the union of the closed neighbourhood of at most $k$ vertices. 

\begin{conj}\label{ConjInducedMinorObsDomination}
    There exists a function $f$ such that, for all $h\in \NN$, every graph $G$ with no $T_h$ induced minor has path domination number at most $f(h)$.
\end{conj}

We provide two pieces of evidence for this conjecture. First, for $K_{1,t}$-free graphs, domination number and independence number of a graph differ by at most a factor of $t$, so \cref{MainThmForest} supports \cref{ConjInducedMinorObsDomination}. Second, a recent result of \citet{NSS2025fat} implies that every graph that excludes a tree as an induced minor has a path decomposition where each bag consists of a bounded number of balls of bounded radius (where the number and size of the balls depend on the excluded tree). The challenge then is to reduce the radius of these balls down to $1$.

{
\fontsize{11pt}{12pt}
\selectfont
	
\hypersetup{linkcolor={red!70!black}}
\setlength{\parskip}{2pt plus 0.3ex minus 0.3ex}

\bibliographystyle{DavidNatbibStyle}
\bibliography{RobReferences}
}

\appendix

\end{document}